\documentclass[11pt,a4paper,twoside]{article}
\usepackage[UKenglish]{babel}
\usepackage[T1]{fontenc}
\usepackage{latexsym,amsfonts,amsmath,amsthm,amssymb,mathrsfs}
\usepackage{fullpage}
\usepackage{xcolor,bm, bbm}
\usepackage{paralist}
\usepackage{todonotes}
\usepackage{dsfont}

\usepackage{fourier}
\newcommand{\xn}{(x_n)_{n \in \mathbb{N}} }
\newcommand{\Mod}[1]{\ (\mathrm{mod}\ #1)}

\DeclareMathOperator{\R}{\mathbb{R}}
\DeclareMathOperator{\N}{\mathbb{N}}
\DeclareMathOperator{\Q}{\mathbb{Q}}

\DeclareMathOperator{\U}{Unif}
\DeclareMathOperator*{\argmax}{arg\,max}
\DeclareMathOperator{\Var}{Var}

\newtheorem{thm}{Theorem}%[section]

\newtheorem{lem}{Lemma}[section]
\newtheorem{cor}[thm]{Corollary}
\newtheorem{df}[lem]{Definition}
\newtheorem{proposition}[lem]{Proposition}
\newtheorem{rem}[lem]{Remark}

\newtheorem*{rem*}{Remark}

\newcommand{\ql}{q_{\ell}}

\allowdisplaybreaks
\title{\bf On the metric upper density of Birkhoff sums for irrational rotations}
\medskip

\author{Lorenz Fr\"uhwirth and Manuel Hauke}

\date{}

\begin{document}

\maketitle

\begin{abstract}
    This article examines the value distribution of $S_{N}(f, \alpha) := \sum_{n=1}^N f(n\alpha)$ for almost every $\alpha$ where $N \in \mathbb{N}$ is ranging over a long interval and $f$ is a $1$-periodic function with discontinuities or logarithmic singularities at rational numbers.
    We show that for $N$ in a set of positive upper density, the order of $S_{N}(f, \alpha)$ is of Khintchine-type, unless the logarithmic singularity is symmetric. Additionally, we show the asymptotic sharpness of the Denjoy\,--\,Koksma inequality for such $f$, with applications in the theory of numerical integration. Our method also leads to a generalized form of the classical Borel-Bernstein Theorem that allows very general modularity conditions.
\end{abstract}

\section{Introduction and main results}

Let $f : \R \to \mathbb{R}$ be $1$-periodic with 
$\int_{[0,1)} |f(x)|\, \mathrm{d}x  < \infty $
and $q \in \mathbb{R}$.
In this article, the object of our interest is \[S_N(f,\alpha,q) := \sum_{n=1}^N f( n \alpha  + q) - N\int_{[0,1)} f(x)\, \mathrm{d}x,\]

which is known as a Birkhoff sum of the irrational circle rotation.
We consider the temporal value distribution along a single orbit of $S_N(f,\alpha,q)$, that is, we fix some initial point $q$ and a rotation parameter $\alpha$, and examine the value distribution of $\{S_N(f,\alpha,q): 1 \leq N \leq M\}$, as $M \to \infty$ for (Lebesgue-) almost every $\alpha$.\\
Since the irrational rotation together with the Lebesgue measure is an ergodic system for all irrational $\alpha$, Birkhoff's ergodic theorem implies that for $1$-periodic $f \in L^1([0,1))$ and almost every $q$, we have
$\lvert S_N(f, \alpha,q)\rvert = o(N).$
If the Fourier coefficients of $f \sim \sum_{n \in \mathbb{Z}} c_n e(nx)$ decay at rate $c_n = O(1/n^2)$
(which holds in particular for $f \in C^2$), then $S_N(f,\alpha,q)$ is bounded for almost every $\alpha$ and all $q \in \mathbb{R}$ (see \cite{quenched,herman}). Thus, the interesting functions to consider are the functions that lack smoothness, in particular functions that have discontinuities or singularities.
\\

In this article, we examine $S_{N}(f,\alpha,q)$ where $q \in \mathbb{Q}$ and all
non-smooth points of $f$ lie at rational numbers. The first class of functions we define is the following (compare to, e.g., \cite{dol_sar_no,quenched}).

\begin{df}[Piecewise smooth functions with rational discontinuities]\label{def_fct1}
    We call a $1$-periodic function $f: \R \to \mathbb{R}$ a
    \textit{piecewise smooth function with rational discontinuities} if there exist $ \nu \geq 1$ and $ 0 \leq x_1 < \ldots < x_{ \nu} < 1$ with $x_i \in \mathbb{Q}, 1 \leq i \leq \nu$  such that the following properties hold:
    \begin{itemize}
        \item $f$ is differentiable on $[0,1)\setminus \{x_1,\ldots,x_{\nu}\}$.
        \item $f'_{\big|[0,1)}$ extends to a function of bounded variation on $[0,1)$.
        \item There exists an $i \in \{1,\ldots,\nu\}$ such that $\lim_{ \delta \rightarrow 0} \left[ f(x_i - \delta) - f(x_i+ \delta) \right] \neq 0$.
    \end{itemize}
\end{df}

This class of functions contains several important representatives such as the sawtooth function $f(x) = \{x\} - 1/2$ 
and the local discrepancy functions with rational endpoints $f(x) = \mathbb{1}_{[a,b]}(\{x\}) - (b-a), a,b \in \mathbb{Q}$. These functions are not only of interest in Discrepancy theory (see \cite{beck2,beck3,schmidt}), but are closely related to the theory of ``deterministic random walks'' (see, e.g., \cite{aake,addo}).

For these local discrepancy functions, a classical theorem of Kesten \cite{kesten_2} shows that $\lvert S_N(f,\alpha,q) \rvert$ is unbounded, since $b - a \notin \mathbb{Z} + \alpha\mathbb{Z}$ for irrational $\alpha$.
In addition to considering essentially smooth $f$, we also examine functions with logarithmic singularities at rational numbers, a class of functions that falls in the framework considered in \cite{dol_sar}.

\begin{df}[Smooth functions with rational logarithmic singularity]\label{def_fct2}
We call a $1$-periodic function $f: \R  \to \mathbb{R}$ with $\int_{[0,1)} f(x) \,\mathrm{d}x = 0$ 
a \textit{smooth function with rational logarithmic singularity} if there exist constants
$c_1,c_2 \in \mathbb{R}$, a $1$-periodic function $ t: \R \to \mathbb{R}$ with bounded variation on $[0,1)$ and $x_1 \in \mathbb{Q}$ such that
\[f(x) = 
\begin{cases}
c_1\log(\lVert x-x_1\rVert ) + c_2\log(\{x-x_1 \}) + t(x), & \text{ if } x \not\equiv x_1 \pmod{1} \\
t(x), & \text{ if } x \equiv x_1 \pmod{1}.
\end{cases} \]
Here and throughout the paper, $\{.\}$ denotes the fractional part and $\lVert . \rVert$ denotes the distance to the nearest integer (for a proper definition see Section \ref{notation}).

If $c_2 = 0$ and $ c_1 \neq 0$, we call the singularity \textbf{symmetric}. If $c_2 \neq 0$, we call it \textbf{asymmetric}.
\end{df}

We examine the maximal and typical oscillations of $S_N(f, \alpha,q)$ for Lebesgue almost every
$\alpha$ where $q \in \mathbb{Q}$ and $f$ is either of the form as in Definition \ref{def_fct1} or Definition \ref{def_fct2}. Our methods give rise to results in two different directions that are elaborated in detail below.

\subsection{Khintchine-type upper density results}

Let us recall the classical Khintchine Theorem from the metric theory of Diophantine approximation.
If $(\psi(q))_{q \in \N}$ is a non-negative sequence, and $q\psi(q)$ is decreasing, then for almost all $\alpha$, the inequality

\[\left\lvert \alpha - \frac{p}{q}\right\rvert < \frac{1}{q \psi(q)} \]
has infinitely many integer solutions $(p, q )\in \mathbb{Z} \times \N $ if and only if $\sum_{k=1}^{\infty}\frac{1}{\psi(k)}$ diverges.

Such a convergence-divergence criterion, often called Khintchine-type result, appears in many different statements
that deal with the metric theory of Diophantine approximation and the closely related theory of continued fractions.
Classical results of this form are, among others, the Borel\,--\,Bernstein Theorem (see Section \ref{prereq}) and another theorem of Khintchine \cite{khin_1923} on the discrepancy of the Kronecker sequence. 
Recall that the discrepancy of a sequence $\xn$ in the unit interval is defined as
\[D_N(\xn) := \sup_{0 \leq a < b \leq 1}\left\lvert\frac{\left |  \{1 \leq n \leq N: x_n \in [a,b)\} \right | }{N} - (b-a)\right\rvert.\]
By \cite{khin_1923}, for an increasing function $ \psi: \R_+ \to \R_+$, one has $D_N((n\alpha)_{n \in \mathbb{N}}) \gg \psi(\log N) + \log N \log \log N$ infinitely often if and only if $\sum_{k =1}^{\infty} \frac{1}{\psi(k)} = \infty$.

Such a Khintchine-type behaviour was also discovered for a certain Birkhoff sum arising from the logarithm of the Sudler product $\prod_{r =1}^{N} 2\lvert \sin(\pi r\alpha) \rvert$.
The analysis of this product led to many interesting developments in various areas in mathematics in the last decades, see, e.g., \cite{sudler_aist_borda,zag_conj,tech_zaf,aj,grepstad_neum,knitan,lubinsky_pade}.
Lubinsky \cite{lubinsky} showed that if $f(x) = \log\left(2 \left\lvert\sin(\pi x)\right\rvert\right)$, then, for almost every $\alpha$ and every monotone increasing $\psi : \R_+ \to \R_+$ with $ \liminf_{k \to \infty}\frac{\psi(k)}{k \log k} = \infty$, the inequalities
\begin{equation}\label{sudler_khin}S_N(f,\alpha,0) \geq \psi(\log N), \quad S_N(f,\alpha,0) \leq -\psi(\log N),\end{equation}
hold for infinitely many $N$, if and only if  $\sum_{k =1}^{\infty} \frac{1}{\psi(k)} = \infty$. This result was refined by Borda \cite{borda} in the following way:
Recall that for a set of integers $A \subseteq \N$, we define its upper density to be
$\limsup_{M \to \infty} \frac{\lvert A_M\rvert}{M}$ where $A_M := \{N \leq M: N \in A\}$. It was shown in  \cite{borda} that in the diverging case, both inequalities \eqref{sudler_khin} hold on a set of positive upper density. This was further improved in \cite{hauke}, where it was shown that the actual upper density equals $1$. Note that $f(x) = \log\left(2 \left\lvert\sin(\pi x)\right\rvert\right)$ does have a single symmetric logarithmic singularity at $0$.
As pointed out in \cite{dol_sar}, the behaviour of Birkhoff sums with $f$ having a symmetric logarithmic singularity is expected to be similar to $f$ being as in Definition \ref{def_fct1}, since the decay of the Fourier coefficients is of the same order. Our first theorem supports this expected behaviour as we obtain a 
Khintchine-type result on the upper density of the same form for piecewise smooth functions with rational discontinuities.

\begin{thm}
\label{ThmUpperDensDiscFunctions}
Let $ \psi : \R_+ \rightarrow \R_+$ be a monotone increasing function, $f$ a function as in Definition \ref{def_fct1}, i.e. $f$ is smooth up to finitely many discontinuities $0 \leq x_1 < \ldots < x_{\nu} < 1$ at rationals, and $q \in \mathbb{Q}$. If $ \sum\limits_{k =1}^{\infty} \frac{1}{\psi(k)} = \infty $, then, for almost all $ \alpha \in [0,1)$, both sets
\begin{equation}
\label{EqSetsUpperDens}
    \left \{ N \in \N \ : \ S_N( f , \alpha ,q) \geq \psi( \log N )  \right\}, \qquad \left \{ N \in \N \ : \ S_N( f , \alpha ,q) \leq - \psi( \log N )  \right\}
\end{equation}
have positive upper density. If $ \sum_{k=1}^{\infty} \frac{1}{\psi(k)} < \infty$, there exists a constant $c > 0$ such that, for almost all $ \alpha \in [0,1)$, the set
\begin{equation*}
    \left \{ N \in \N  :  \lvert S_N( f, \alpha, q) \rvert \geq  \psi( \log N) + c \log N \log \log N  \right \}
\end{equation*}
is finite. 
\end{thm}

Theorem \ref{ThmUpperDensDiscFunctions} shows the following interesting behaviour:
Let $\psi$ be monotonically increasing and \\$\liminf_{k \to \infty} \frac{\psi(k)}{k \log k} = \infty$. As soon as there are infinitely many $N \in \N$ such that $\lvert S_N( f , \alpha ,q) \rvert \geq \psi( \log N )$, then immediately a ``positive proportion'' of all natural numbers satisfies this property.\newline

\par{}

In Section \ref{SubsectionProofThmDiscFunctions}, we prove Theorem \ref{ThmUpperDensDiscFunctionsRefined}, which is a slightly stronger result than Theorem \ref{ThmUpperDensDiscFunctions}. It turns out that for many classical functions that fall in the framework of Definition \ref{def_fct1}, the sets \eqref{EqSetsUpperDens} have actually upper density $1$.

\begin{cor}\label{special_fct_dens_1}
    Let $\psi$ and $q$ be as in Theorem \ref{ThmUpperDensDiscFunctions}
    with $\sum_{k =1}^{\infty}\frac{1}{\psi(k)} = \infty$
    . Then, for almost all $\alpha \in [0,1)$, the sets in \eqref{EqSetsUpperDens} have upper density $1$ for the following functions $f$:

    \begin{itemize}
    \item The classical saw-tooth function $f(x) = \{x\} - 1/2$.
    \item The local discrepancy functions with rational endpoints $f(x) = \mathbb{1}_{[a,b]}(\{x\}) - (b-a)$. In particular, this means that the set $ \left \{ N \in \N : D_N( (n \alpha)_{n \in \N} ) \geq \psi( \log N) \right \}$ has upper density $1$, improving the result in \cite{khin_1923}.
    \item the $1/2$-discrepancy function $f(x) = \mathds{1}_{[0,1/2)}(\{x\}) - \mathds{1}_{[1/2,1)}(\{x\})$, which is intensively studied in, e.g., \cite{ralston1,ralston2}.
    \end{itemize}
\end{cor}

Naturally, the question arises whether a similar behaviour can be expected for functions with logarithmic singularities. Note that in many applications (see for example \cite{dol_fay} and the references therein), one is not only interested in symmetric, but also asymmetric singularities. In the next statement we show that for functions $f$ with an asymmetric singularity, there is an analogue of Theorem \ref{ThmUpperDensDiscFunctions} with an additional scaling factor of $\log N$. 

\begin{thm}\label{logthm}
Let $f$ be as in Definition \ref{def_fct2}. Then, for any non-decreasing  $\psi: \R_+ \to \mathbb{R}_{+}$ and for any $q \in \mathbb{Q}$, the following holds:
\begin{enumerate}
    \item[(i)]
If the logarithmic singularity is \textit{asymmetric} and $\sum\limits_{k=1}^{\infty}\frac{1}{\psi(k)} = \infty$, then, for almost every $\alpha \in [0,1)$, we have that both sets
\begin{equation}\label{log_sets}
 \quad    \left \{ N \in \N : S_N( f, \alpha,q) \geq \log N \psi( \log N ) \right \}, \qquad 
 \quad \left \{  N \in \N :  S_N(f,  \alpha,q) \leq - \log N \psi( \log N ) \right \} 
\end{equation}
have upper density $1$. If $\sum_{k=1}^{\infty} \frac{1}{\psi(k)} < \infty $, then there exists a constant $c > 0$ such that, for almost all $ \alpha \in [0,1)$, the set
\begin{equation*}
    \left \{ N \in \N  :  \lvert S_N( f, \alpha, q) \rvert \geq \log N \left( \psi( \log N) + c \log N \log \log N  \right) \right \}
\end{equation*}
is finite.
\item[(ii)] If the logarithmic singularity is \textit{symmetric}, then for almost every $\alpha \in [0,1)$, we have 
\[\lvert S_N(f, \alpha , q) \rvert \ll (\log N)^2\log \log N.\]
\end{enumerate}
\end{thm}

Theorem \ref{logthm} shows the surprising fact that one should expect a completely different oscillation between the Birkhoff sums of functions with symmetric and asymmetric singularity. To see this, we note that $ \psi : \R_+ \to \R_+ $ with $\psi(x) := x \log (x +1) \log \log (x +10)$ satisfies $ \sum_{k=1}^{\infty} \frac{1}{ \psi(k)} = \infty$ and thus, for a function $f$ with an asymmetric logarithmic singularity, the sets in \eqref{log_sets} have upper density $1$. However, the same sets are finite if the underlying Birkhoff sum is generated by a function $g$ with symmetric logarithmic singularity. This is because Theorem \ref{logthm} (ii) tells us that $ \lvert S_N(f, \alpha , q) \rvert \ll (\log N)^2\log \log N = o( \log N \psi(\log N) )$.

\begin{rem}\item
\begin{itemize}
\item One can clearly see from the proof of Theorem \ref{logthm}(i) that it is possible to generalize the result to functions of the form $f(x) = f_1(x) + f_2(x) + f_3(x)$, where $f_1(x)$ is a smooth function with asymmetric singularity at a rational number $x_1$, $f_2$ is a smooth function with finitely many symmetric logarithmic singularities at rational numbers $x_2,x_3,\ldots, x_{n}$ and $f_3$ is a piecewise smooth function with finitely many discontinuities.

\item The actual maximal oscillation of Birkhoff sums with a symmetric logarithmic singularity for generic $\alpha$ remains open.
Using \cite[Proposition 12]{borda}, the result of \cite{lubinsky} and $f$ being a smooth $1$-periodic function with its symmetric logarithmic singularity located at $0$, we have that, for almost every $\alpha \in [0,1)$, $\lvert S_N(f,\alpha,0) \rvert \ll \psi(\log N) + \log N \log \log N$ for any monotone $\psi$ with $\sum_{k=1}^{\infty} \frac{1}{\psi(k)} < \infty.$
Most probably, the bound $(\log N)^2\log \log N$ is not sharp for logarithmic singularities that are located at arbitrary rationals. We did not aim to achieve the best possible bound, but wanted to
stress the different behaviour of symmetric and asymmetric logarithmic singularities.
Possibly, the upper bound for the Birkhoff sum with symmetric logarithmic singularity coincides with the Khintchine-type behaviour in Theorem \ref{ThmUpperDensDiscFunctions}.
A proof of this would probably need to make use of delicate estimates on shifted cotangent sums and is beyond the scope of this paper.
\end{itemize}
\end{rem}

\subsection{Sharpness of the Denjoy--Koksma inequality}\label{koksma_sec}

Recall the classical Denjoy--Koksma inequality (see, e.g., \cite{herman}):
Let $\alpha$ be fixed and let $p_n/q_n$ denote its $n$-th convergent. If $f$ is a $1$-periodic function of bounded variation $\Var(f)$ on $[0,1)$, then, for any $q \in \R$,
\[ \left \lvert S_{q_n}(f, \alpha, q) \right \rvert  = \left\lvert\sum_{k=0}^{q_n -1} f(k\alpha+q) - q_n\int_{0}^{1} f(x)\,\mathrm{d}x\right\rvert
\leq \Var(f).\]

For $N = \sum_{i=0}^{K(N)-1} b_iq_i$ being its Ostrowski expansion (see Section \ref{cf_prop} for details), we immediately obtain the bound

\begin{equation}
\label{koksma_bound}
 \lvert S_{N}(f,\alpha,q)\rvert
\leq \Var(f) \sum_{i =1}^{K(N)}a_i,
\end{equation}

where $\alpha = [0;a_1,a_2,\ldots]$ is the classical continued fraction expansion
and $K(N)$ denotes the integer $K$ such that $q_{K-1} \leq N < q_K$.
It is natural to ask whether \eqref{koksma_bound} is sharp for particular functions $f$. This was essentially already proven in \cite{khin_1923} for both the classical saw-tooth function $\{x\} - 1/2$ and the local discrepancy functions $\mathds{1}_{[a,b]}( \{x\}) - (b-a)$: for almost every $\alpha$, there are infinitely many $N$ where \eqref{koksma_bound} can be reverted up to an absolute positive constant. However, to the best of our knowledge, all results in this direction only show
that this bound is essentially obtained for infinitely many $N$, but there is no statement about the frequency of those $N$. This is shown in the following Theorem.

\begin{thm}\label{koks_sharp}
    Let $f$ be a function with finitely many discontinuities at rationals (see Definition \ref{def_fct1}) and let $q \in \mathbb{Q}$. For fixed $\alpha $ and $N \in \N$, let $K(N)$ denote
     the integer $K$ such that $q_{K-1} \leq N < q_K$.
Then, for almost all $ \alpha \in [0,1)$, both sets
\begin{equation*}
    \left \{ N \in \N \ | \ S_N( f , \alpha, q ) \gg  \sum_{i =1}^{K(N)}a_i \right\}, \qquad \left \{ N \in \N \ | \ S_N( f , \alpha ,q) \ll - \sum_{i =1}^{K(N)}a_i  \right\}
\end{equation*}
have positive upper density. The implied constants depend on $\alpha,f$ and $q$. 
\end{thm}

Analogously to Corollary \ref{special_fct_dens_1}, the following stronger version of Theorem \ref{koks_sharp} can be obtained.

\begin{cor}\label{koks_dens1}
If $f(x) = \{x\} - 1/2$ or $ f(x) = \mathbb{1}_{[a,b]}(\{x \}) - (b-a)$ we have the following:
For almost all $\alpha \in [0,1)$ and any $\,0 < r < 1$, there exists a constant $C(r) = C(r,f,q) > 0$ such that both sets
\begin{equation*}
    \left \{ N \in \N \ \Big| \ S_N( f , \alpha,q) \geq C(r)  \sum_{i =1}^{K(N)}a_i \right\}, \qquad \left \{ N \in \N \ \Big| \ S_N( f , \alpha , q) \leq - C(r) \sum_{i =1}^{K(N)}a_i \right\}
\end{equation*}
have upper density at least $r$.
\end{cor}

\begin{rem}
Note that in contrast to Corollary \ref{special_fct_dens_1}, we cannot include $r = 1$ in Corollary \ref{koks_dens1} since our method of proof gives $\lim\limits_{r \to 1} C(r) = 0$. This is due to the fact that the convergence-divergence condition in the Khintchine formulation enables to
hide a sufficiently large constant in the divergence condition of $\psi$ (since for $\tilde{\psi}(k) := c\psi(k)$, $\sum_{k=1}^{\infty}\frac{1}{\tilde{\psi}(k)}$ converges if and only if $\sum_{k=1}^{\infty}\frac{1}{{\psi}(k)}$ does)
, which is not possible to do in this setting.
\end{rem}

Theorem \ref{koks_sharp} has consequences in the theory of numerical integration: 
Assuming that there are functions $f$ such that the bound \eqref{koksma_bound}
is attained only along a very sparse subsequence $(N_k)_{k \in \mathbb{N}}$, one could hope with a randomized approach to hit this sequence very rarely. 
Then, one could generate by some (randomized) algorithm an increasing sequence of integers $(M_j)_{j \in \mathbb{N}}$ and consider an irrational $\alpha$ drawn uniformly at random from the unit interval. With high probability, one would expect $\left\lvert S_{M_j}(f,a,q)\right\rvert = o\left(\sum\limits_{i=1}^{K(M_j)} a_i\right)$. 
Theorem \ref{koks_sharp} implies that such an approach will most likely fail. It shows that, almost surely, a positive proportion of those $M_j$ satisfies
\[\lvert S_{M_j}(f,a,q)\rvert \gg \sum_{i=1}^{K(M_j)} a_i
\gg \log M_j\log \log M_j.\] This implies that every low-discrepancy sequence (such as the Kronecker sequences $(\{n\alpha\})_{n \in \mathbb{N}}$ where $\alpha$ is badly approximable) gives a better error bound in numerical integration, regardless of the support of the function $f$ and the chosen algorithm to generate the sequence $(M_j)_{j \in \mathbb{N}}$.\\

Next, assume that $f$ has a singularity, but is still integrable. The singularity makes an application of the classical Denjoy--Koksma inequality impossible since the variation of $f$ is not bounded. However, we can still get a nontrivial bound on $S_{N}(f, \alpha,q)$ provided that the orbit
$\{q + n\alpha: 1 \leq n \leq N\}$ stays away from the singularity. 

\begin{proposition}[Denjoy--Koksma inequality with singularity]
\label{denj_koks_log}
Let $N \in \N, q \in [0,1)$ arbitrary and let $N = \sum_{i=0}^{K(N)-1} b_iq_i$ be its Ostrowski expansion. Assume that $f$ is a $1$-periodic function with a single singularity in $x_1$. Let
$x_1 \in A_N \subseteq [0,1)$, where $A_N$ is an interval (modulo $1$) such that $\{q + n\alpha- x_1\} \notin A_N$ for all $1 \leq n \leq N$. Then,
\begin{equation}
\label{koksma_bound_sing} 
\left\lvert S_{N}(f,\alpha,q)\right\rvert
\ll \sup_{x \in [0,1)\setminus A_N} \lvert f(x)\rvert  \sum_{i =0}^{K(N) -1}b_i
+ N \left\lvert\int_{A_N} f(x) \,\mathrm{d}x \right \rvert.
\end{equation}
In particular, we have, for $K \in \N$, 
\begin{equation*}
    \max_{1 \leq N \leq q_K} | S_N(f, \alpha, q) | \ll \sup_{x \in [0,1)\setminus A_{q_K}} \lvert f(x)\rvert  \sum_{i =1}^{K}a_i+ q_K  \left\lvert\int_{A_{q_{K}}} f(x) \,\mathrm{d}x\right\rvert.
\end{equation*}
\end{proposition}

This bound is not new but was used already in, e.g., \cite{jito,jitoIII, knill_les}.
The statement follows immediately by defining
\[\Tilde{f}(x) := \begin{cases}
f(x) &\text{ if } x \in [0,1)\setminus A_N,
\\ 0 &\text{ otherwise},
\end{cases}\]

and applying the classical Denjoy--Koksma inequality \eqref{koksma_bound} to $\Tilde{f}$.

The following theorem shows that for asymmetric logarithmic singularities, the estimate \eqref{koksma_bound_sing}
is also sharp in the sense of Theorem \ref{koks_sharp},
which in particular extends the consequences for numerical integration to functions with an asymmetric logarithmic singularity.

\begin{thm}\label{koks_sharp_log}
    Let $f$ be a $1$-periodic smooth function with rational asymmetric logarithmic singularity in $x_1$ as in Definition \ref{def_fct2}. For fixed $\alpha \in [0,1), q \in \mathbb{Q}$ and $N \in \mathbb{N}$, we denote by
    $K(N)$ the integer $K$ such that $q_{K-1} \leq N < q_K$. Further, let $A_N = \left (x_1 - g(N),x_1 + g(N) \right )$ be an interval with $g(N) = \min_{n \leq N} \lVert n\alpha + q - x_1 \rVert$.
Then, for almost all $ \alpha \in [0,1)$ and any $0 < r < 1$, there exists a constant $C(r) = C(r,f,q) > 0$ such that both sets
\begin{equation*}\begin{split}
    &\left \{ N \in \N \ | \ S_N( f , \alpha, q ) \geq C(r)  \sup_{x \in [0,1)\setminus A_N}\lvert f(x)\rvert  \sum_{i =1}^{K(N)}a_i
+ N \left\lvert\int_{A_N} f(x)\,\mathrm{d}x\right\rvert \right\},\\ &\left \{ N \in \N \ | \ S_N( f , \alpha ,q) \leq - C(r)\sup_{x \in [0,1)\setminus A_N}\lvert f(x)\rvert  \sum_{i =1}^{K(N)}a_i
- N \left\lvert\int_{A_N} f(x) \,\mathrm{d}x\right\rvert \right\}
    \end{split}
\end{equation*}
have upper density of at least $r$.
\end{thm}

\section{Prerequisites}\label{prereq}
\subsection{Notation}\label{notation}

Given two functions $f,g:(0,\infty)\to \mathbb{R},$ we write $f(t) = O(g(t)), f \ll g$ or $g \gg f$ if
$\limsup_{t\to\infty} \frac{|f(t)|}{|g(t)|} < \infty$. Any  dependence of the value of the limes superior above on potential parameters is denoted by appropriate subscripts. For two sequences $ (a_k)_{k \in \N}$ and $(b_k)_{k \in \N}$ with $b_k \neq 0$ for all $k \in \N$, we write $a_k \sim b_k, k \to \infty$, if $ \lim_{k \to \infty} \frac{a_k}{b_k}=1$.
Given a real number $x\in \mathbb{R}$, we write $\{x\}= x - \lfloor x \rfloor$ for the fractional part of $x$ and $\|x\|=\min\{|x-k|: k\in\mathbb{Z}\}$ for the distance of $x$ to its nearest integer.
We denote the characteristic function of a set $A$ by $\mathds{1}_A$ and understand
the value of empty sums as $0$. We denote the cardinality of a set $A \subseteq \N$ as $ \lvert A \rvert $. For shorter notation, we write $S_N(f,\alpha) := S_N(f,\alpha,0)$. Let $ \sigma \big( X_1, X_2, \ldots \big)$ denote the $\sigma$-algebra generated by random variables $X_1, X_2, \ldots $.

\subsection{Continued fractions}\label{cf_prop}
In this subsection, we collect all classical results from the theory of continued fractions that we need to prove our main results. Every irrational $\alpha \in \R$ has a unique infinite continued fraction expansion denoted by $[a_0;a_1,a_2,\dots ]$ with convergents $p_k/q_k := [a_0;a_1, \dots, a_k]$ that satisfy the recursions
\begin{equation*}
 p_{k+1}  = p_{k+1}(\alpha) = a_{k+1}(\alpha) p_k + p_{k-1}, \qquad q_{k+1}  = q_{k+1}(\alpha) = a_{k+1}(\alpha) q_k + q_{k-1}, \quad k \in \N,
\end{equation*}
with initial values $p_0 = a_0,\; p_1 = a_1a_0 +1,\; q_0 = 1,\; q_1 = a_1$.
For the sake of brevity, we just write $a_k, p_k, q_k$, although these quantities depend on $\alpha$. We know that $p_k/q_k$ are good approximations for $\alpha $ and satisfy the inequalities
\begin{equation*}
\frac{1}{q_{k+1} + q_k} \leq \delta_k := | q_k \alpha - p_k | \leq \frac{1}{q_{k+1}}, \quad k \in \N.
\end{equation*}

Fixing an irrational $\alpha$, the Ostrowski expansion of a non-negative integer $N$ is the unique representation
\begin{equation*}
    N= \sum_{\ell=0}^{K-1} b_{\ell} q_{\ell},
\end{equation*}
where $b_{K-1} \neq 0$, $0 \leq b_0 < a_1$, $ 0 \leq b_{\ell} \leq a_{\ell+1} $ for $ 1 \leq \ell \leq K-1$  and if $b_{\ell} = a_{\ell}$, then $ b_{\ell-1}= 0$.

So far all statements can be made for arbitrary irrational numbers, but since this article considers the almost sure behaviour, we make use of the well-studied area of the metric theory of continued fractions. We state several classical results below which hold for almost every $\alpha \in \R$. We will use them frequently in the proofs of our results.

\begin{itemize}
    \item The Borel-Bernstein Theorem (\cite{bernstein}, see also \cite{bugeaud} for a more modern formulation without monotonicity condition): For any non-negative function $\psi: \mathbb{\R_+ } \to \mathbb{R}_{+}$, we have
    \begin{equation}
        \label{Eqbernstein}
        \left | \left\{k \in \mathbb{N}: a_k > \psi(k)\right\} \right | \text{ is } \begin{cases} \text{ infinite }\hspace{5mm}&\text{if } \sum\limits_{k = 0}^{\infty} \frac{1}{\psi(k)} = \infty,
        \\ \text{ finite} \hspace{5mm} &\text{if }\sum\limits_{k = 0}^{\infty} \frac{1}{\psi(k)} < \infty.\end{cases}
    \end{equation}
\item (Diamond and Vaaler \cite{diamond_vaaler}):
\begin{equation}
\label{Eq_trimmed_sum}
\sum_{\ell \leq K} a_{\ell} - \max\limits_{\ell \leq K} a_{\ell} \sim \frac{K \log K}{\log 2}, \quad K \to \infty.
\end{equation} 
\item (Khintchine and L\'{e}vy, see, e.g., \cite[Chapter 5, §9, Theorem 1]{rock_sz}):
\begin{equation}
\label{Eq_size_of_q_k}
\log q_k \sim \tfrac{\pi^2}{12 \log 2} k \text{ as } k \to \infty.
\end{equation}
\end{itemize}

\section{Functions with discontinuities}
\label{SubsectionProofThmDiscFunctions}

We start this section with a decomposition lemma that is of a similar form to \cite[Appendix A]{quenched}. %It enables us to decompose any piecewise smooth function into a linear combination of shifted sawtooth and indicator functions.

\begin{lem}[Decomposition Lemma]
\label{LemRepresentationf}
Let $f : \R \rightarrow \R$ be as in Definition \ref{def_fct1}, i.e. $f$ is piecewise smooth with (possible) discontinuities at $0 \leq x_1 < \ldots < x_{\nu} < 1$. 
Let 
$g:\R \to \R$ be defined as \[g(x) = \left ( \{x\} - \frac{1}{2} \right) \sum_{i=1}^{\nu} A_i + \sum_{i=1}^{\nu} c_i \left(  \mathbb{1}_{[x_{i-1}, x_i)}(x) -(x_i - x_{i-1} ) \right),\]
where 
$A_i := \lim_{ \delta \rightarrow 0} \left[f(x_i - \delta) - f(x_i + \delta) \right]$, $c_i := \sum_{j=1}^{i}A_j$ and $x_{0} := 0$.
Then, for any $N \in \N, q \in \Q$ and for almost every $\alpha$, we have
\[S_N(f,\alpha,q) = S_N(g,\alpha,q) + O(1).\]

If $f$ satisfies
\begin{equation}\label{aici_condition} \sum_{i =1}^{\nu}A_i \neq 0 \text{ or } \sum_{i =1}^{\nu}c_i(x_{i}-x_{i-1}) \neq 0,\end{equation}
then also $f_y(x) := f(x+y)$, for any $y \in [0,1)$, satisfies \eqref{aici_condition}.
\end{lem}

\begin{proof}
We define $\varphi : \R \rightarrow \R$ as $ \varphi(x) := f(x) - \sum_{i=1}^{\nu} A_i \left( \{ x - x_i\} - \frac{1}{2} \right)$. 
Further, we see $ \{ x - x_i \} - \frac{1}{2} = \{ x \} - \frac{1}{2} - x_i + \mathbb{1}_{[0, x_i)} (x)$ and hence
\begin{align*}
f(x) &= \sum_{i=1}^{\nu}  A_i \left( \{ x - x_i \} - \frac{1}{2} \right) + \varphi(x)   \\
     &= \left( \{ x  \} - \frac{1}{2} \right) \sum_{i=1}^{\nu}  A_i  + \sum_{i=1}^{\nu}  A_i \left( \mathbb{1}_{[0, x_i)} (x) - x_i \right)  +  \varphi(x) \\
     &= \left( \{ x  \} - \frac{1}{2} \right) \sum_{i=1}^{\nu}  A_i  + \sum_{i=1}^{\nu}  c_i \left( \mathbb{1}_{[x_{i-1}, x_i)} (x) -(x_i - x_{i-1} ) \right)  +  \varphi(x)
     \\& = g(x) + \varphi(x).
\end{align*}
By the choice of $f$, there exists a 
function $\psi$ that is differentiable with $\psi'$ being of bounded variation and
 $\psi_{\big|[0,1)\setminus \{x_1,\ldots,x_{\nu}\}} = \varphi_{\big|[0,1)\setminus \{x_1,\ldots,x_{\nu}\}}$. By the properties of $\psi$, we get $S_N(\psi,\alpha,q) = O(1)$ (see Appendix A in \cite{quenched}). Since $\alpha$ is irrational and $q \in \Q$,
 we have $\psi(n\alpha + q) = \varphi(n\alpha + q)$ for any $n \in \N$ and thus,
 \[S_N(f,\alpha,q) = S_N(g,\alpha,q) + S_N(\psi,\alpha,q) =  S_N(g,\alpha,q) + O(1),\] 

which proves the first part of the statement.
For the second part, one sees immediately that $\sum_{i =1}^{\nu}A_i$ is invariant under translation. By a slightly longer, but elementary calculation, one finds that under the assumption of $\sum_{i =1}^{\nu}A_i = 0$, also $\sum_{i =1}^{\nu}c_i(x_{i}-x_{i-1}) \neq 0$ is invariant under translation.
\end{proof}

The definition of the quantities $A_i,c_i$ in Lemma \ref{LemRepresentationf} allows us to state stronger, but more technically involved versions of Theorem \ref{ThmUpperDensDiscFunctions} respectively Theorem \ref{koks_sharp}. This refinement will also immediately imply Corollary \ref{special_fct_dens_1} and Corollary \ref{koks_dens1}.

\begin{thm}
\label{ThmUpperDensDiscFunctionsRefined}
Let $ \psi : \R_+ \rightarrow \R_+$ be a monotonically increasing function with $ \sum_{k =1}^{\infty}\frac{1}{\psi(k)} = \infty$ and let $ f : \mathbb{R} \rightarrow \R$ be as in Definition \ref{def_fct1}, i.e. $f$ is essentially smooth with (possible) discontinuities at finitely many rationals $ 0 \leq x_1 < \ldots < x_{ \nu} < 1$. Let $A_i = \lim_{ \delta \rightarrow 0} \left[ f(x_i - \delta) - f(x_i+ \delta) \right]  $, $ c_i = \sum_{j=1}^i A_j$ for $i \in \{1, \ldots , \nu \} $ and set $x_0 =0$. If $\;\sum\limits_{i = 1}^{ \nu} A_i \neq  0 $ or $ \sum\limits_{i=1}^{ \nu} c_i ( x_{i} - x_{i-1}) \neq 0$, for almost all $\alpha \in [0,1)$ and all $q \in \mathbb{Q}$, both sets
\begin{equation}
\label{EqSetsUpperDensThmRefined}
\left \{ N \in \N \ | \ S_N(f, \alpha,q) \geq \psi( \log N )  \right \}, \qquad \left \{ N \in \N \ | \ S_N(f, \alpha,q) \leq - \psi( \log N )  \right \}
\end{equation}
have upper density $1$. 
\par{}
If $\;\sum_{i = 1}^{ \nu} A_i = 0 $ and $ \sum_{i=1}^{ \nu} c_i ( x_{i} - x_{i-1}) = 0$, for almost all $\alpha \in [0,1)$ and all $q \in \mathbb{Q}$, both sets in \eqref{EqSetsUpperDensThmRefined} have positive upper density.  
\end{thm}

\begin{thm}
\label{ThmKoksRefined}
Let $ \psi,f, \nu , A_i$ and $ c_i$ for $i \in \{1, \ldots, \nu\}$ be as in Theorem \ref{ThmUpperDensDiscFunctionsRefined}. If $\,\sum\limits_{i = 1}^{ \nu} A_i \neq  0 $ or $\;\sum\limits_{i=1}^{ \nu} c_i (x_{i} - x_{i-1}) \neq 0$, for almost all $\alpha \in [0,1)$ and all $q \in \mathbb{Q}$, we have the following:

For any $\,0 < r < 1$, there exists a constant $C(r) = C(r,f,q) > 0$ such that both sets
\begin{equation*}
    \left \{ N \in \N \ \Big| \ S_N( f , \alpha,q ) \geq C(r)  \sum_{i =1}^{K(N)}a_i \right\}, \qquad \left \{ N \in \N \ \Big| \ S_N( f , \alpha ,q) \leq - C(r) \sum_{i =1}^{K(N)}a_i \right\}
\end{equation*}
have upper density of at least $r$. Here $K(N)$ denotes the integer such that
$q_{K(N)-1} \leq N < q_{K(N)}$.
\par{}
If $\;\sum_{i = 1}^{ \nu} A_i = 0 $ and $\;\sum_{i=1}^{ \nu} c_i ( x_{i} - x_{i-1}) = 0$, for almost all $\alpha \in [0,1)$, there exist constants $r_0 = r_0(f,\alpha,q) > 0$ and $C = C(f,\alpha,q) >0$ such that both sets
\begin{equation*}
    \left \{ N \in \N \ \Big| \ S_N( f , \alpha,q ) \geq C  \sum_{i =1}^{K(N)}a_i \right\}, \qquad \left \{ N \in \N \ \Big| \ S_N( f , \alpha,q ) \leq - C \sum_{i =1}^{K(N)}a_i \right\}
\end{equation*}
have upper density of at least $r_0$.
\end{thm}

Naturally, the question arises whether the conditions
$\sum_{i =1}^{\nu} A_i \neq 0$ or $\sum_{i=1}^{\nu} c_i(x_{i}-x_{i-1}) \neq 0$ give an exact
characterization of functions that satisfy the statements in Theorem \ref{ThmUpperDensDiscFunctionsRefined} and Theorem \ref{ThmKoksRefined}. We show that these assumptions are not necessary, but without any condition on the interplay of the location of the discontinuities and their jump heights, one cannot hope to achieve upper density $1$. In fact, we provide two classes of functions that, in general, do not satisfy $\sum_{i =1}^{\nu} A_i \neq 0$ or $\sum_{i=1}^{\nu} c_i(x_{i}-x_{i-1}) \neq 0$; for one class, the sets
\begin{equation*}
\left \{ N \in \N \ | \ S_N(f, \alpha,q) \geq \psi( \log N )  \right \} \qquad \left \{ N \in \N \ | \ S_N(f, \alpha,q) \leq - \psi( \log N )  \right \}
\end{equation*}
both have upper density $1$ for any $q \in \mathbb{Q}$. However for the other class of functions, this fails to hold. The proofs of both these statements (Proposition \ref{upp_dens_wo_cond} and Proposition \ref{counterex_density1}) can be found in the Appendix.

\begin{proposition}\label{upp_dens_wo_cond}
    Let 
    \begin{equation}\label{counterexample_form}
    f(x) = \mathds{1}_{ \left [\frac{r_1}{s_1},\frac{r_2}{s_2} \right ]}( \{x \} ) - \mathds{1}_{ \left [\frac{r_2}{s_2},\frac{r_3}{s_3} \right ]}( \{x \} ), \quad r_i,s_i \in \mathbb{Z} \setminus \{0\}, \gcd (r_i,s_i) = 1, \quad r_1^2 + r_3^2 > 0.
    \end{equation}
    Then, for every $q \in \mathbb{Q}$ and almost all $\alpha \in [0,1)$, both sets
\begin{equation}
\left \{ N \in \N \ | \ S_N(f, \alpha,q) \geq \psi( \log N )  \right \}, \qquad \left \{ N \in \N \ | \ S_N(f, \alpha,q) \leq - \psi( \log N )  \right \}
\end{equation}
have upper density $1$. 
\end{proposition}
\begin{rem}
Let $f(x)= \mathds{1}_{ \left [\frac{1}{4},\frac{1}{2} \right ]}(\{ x \} ) - \mathds{1}_{\left [\frac{1}{2},\frac{3}{4} \right ]}( \{x \})$ which satisifes the assumptions of Proposition \ref{upp_dens_wo_cond}. One can easily check that $ \sum_{i=1}^{ \nu} A_i = \sum_{i =1}^{\nu} c_i ( x_i - x_{i-1} ) = 0$. Theorem \ref{ThmUpperDensDiscFunctionsRefined} (and hence Theorem \ref{ThmUpperDensDiscFunctions}) only shows that the sets in \eqref{EqSetsUpperDens} have positive upper density. Proposition \ref{upp_dens_wo_cond} shows that we have indeed upper density $1$ meaning that the assumptions in Theorem \ref{ThmUpperDensDiscFunctionsRefined} are not sharp.
\end{rem}

The following proposition shows that there are functions of the form as in Definition \ref{def_fct1} where the sets in \eqref{EqSetsUpperDens} do not have upper density $1$. 
\begin{proposition}\label{counterex_density1}
    Let $u,v,w$ be positive integers with $u < v < w/2$, $\gcd(u,w) = \gcd(v,w) = 1$ and let
    \[f(x) = \mathds{1}_{\left[\frac{u}{w},\frac{v}{w}\right]} ( \{x \})
    - \mathds{1}_{\left[ 1- \frac{v}{w}, 1 -\frac{u}{w}\right]} ( \{x \}) .\]
Then, there exists a monotone increasing function $ \psi : \R_+ \rightarrow \R_+$ with $ \sum_{k=1}^{\infty} \frac{1}{\psi(k)} = \infty$ and a constant $0 < \delta  < 1$ such that, for almost every $\alpha \in [0,1)$, the set
\begin{equation*}
    \left \{ N \in \N : S_N( f, \alpha) \geq \psi( \log N) \right\} 
\end{equation*}
has upper density of at most $ 1- \delta $.
\end{proposition}

\begin{proof}[Proof of Theorem \ref{ThmUpperDensDiscFunctions}, Theorem \ref{koks_sharp}, Corollary \ref{special_fct_dens_1} and Corollary \ref{koks_dens1}]
Corollary \ref{special_fct_dens_1} follows immediately from Theorem \ref{ThmUpperDensDiscFunctionsRefined} since all functions considered in Corollary \ref{special_fct_dens_1} satisfy $\;\sum\limits_{i = 1}^{ \nu} A_i \neq  0 $ or $ \sum\limits_{i=1}^{ \nu} c_i ( x_{i} - x_{i-1}) \neq 0$. Analogously, Theorem \ref{koks_sharp} and Corollary \ref{koks_dens1} follow directly from Theorem \ref{ThmKoksRefined}. Theorem \ref{ThmUpperDensDiscFunctionsRefined} implies the statement in Theorem \ref{ThmUpperDensDiscFunctions} for the case $ \sum_{k=1}^{\infty} \frac{1}{ \psi(k) } = \infty$,
so we are left with the case $\sum_{k=1}^{\infty} \frac{1}{ \psi(k) } < \infty $.
By the Denjoy--Koksma inequality \eqref{koksma_bound}, we have
\begin{equation*}
    \lvert S_N(f, \alpha, q) \rvert \leq \Var(f) \left( \sum_{i=1}^{K(N)} a_i \right),
\end{equation*}
where $q_{K(N)-1} \leq N < q_{K(N)}$. We have $ \sum_{i=1}^{K(N)} a_i  = a_{K_0} +  \sum_{i=1, i \neq K_0}^{K(N)} a_i $ with $K_0 = \argmax_{i=1, \ldots, K(N)} a_i$. 
We define $\tilde{\psi} (k) := c_1 \psi( c_2 k)$ for constants $c_1, c_2 > 0$ specified later. Since $\psi$ is monotone, it follows immediately that $ \sum_{k=1}^{\infty} \frac{1}{\tilde{\psi}(k) } < \infty$. Thus, \eqref{Eqbernstein} implies that, for sufficiently large $N$, we have $ a_{K_0} \leq \tilde{\psi}(K_0) $. Moreover, by \eqref{Eq_trimmed_sum} there exists an absolute constant $ \tilde{c} > 0$ such that $ \sum_{i=1, i \neq K_0}^{K(N)} a_i \leq \tilde{c} K(N) \log K(N) $. Together, we get 
\begin{align*}
\lvert S_N(f, \alpha , q) \rvert & \leq \Var(f) \left( \tilde{\psi}(K_0) + \tilde{c} K(N) \log K(N) \right) \\
& \leq \psi( \log N) + c \log N \log \log N,
\end{align*}
where we choose $c_1$ and $c_2$ such that 
\[  \Var (f) \tilde{\psi}(K_0 ) \leq \psi( \log N)\]
and $c > 0 $ in a way that $ \tilde{c}  \Var(f) K(N) \log K(N) \leq c \log N \log \log N$. Thus, there exists a constant $c > 0$ such that, for almost all $\alpha \in [0,1)$, there are only finitely many $N \in \N$ with
\begin{equation*}
\lvert S_N(f, \alpha, q) \rvert \geq \psi( \log N) + c \log N \log \log N.
\end{equation*} 
\end{proof}

\subsection{Heuristic of the proofs}

We will briefly line out the main ideas of the proof of Theorems \ref{ThmUpperDensDiscFunctionsRefined} and \ref{ThmKoksRefined}. One of the core tools we are using is the well-known result in metric number theory that for almost every $\alpha \in [0,1)$, there are infinitely many $K \in \N$ such that $a_K$ dominates the sum of the preceding partial quotients, that is 
$\sum_{i=1}^{K-1} a_i = o(a_K)$. Here, $K$ will always satisfy this property. For $q_{K-1} < N < q_K$ and $N = b_{K-1}q_K + N'$, $N' < q_{K-1}$, we first get rid of $S_{N'}(f,\alpha)$ by an application of the Denjoy--Koksma inequality. The rest of the proof is to show that essentially $S_{b_{K-1}q_{K-1}}(f,\alpha) \gg a_K$ for most $N$, which then implies morally both Theorems \ref{ThmUpperDensDiscFunctionsRefined} and \ref{ThmKoksRefined} by an application of the Borel-Bernstein Theorem. 
As $\{n\alpha\}_{n =1}^{q_{K-1}}$ is close to being uniformly distributed in the unit interval, it is natural to analyze
\[S_{(b+1)q_{K-1}}(f,\alpha) - S_{bq_{K-1}}(f,\alpha)
= \sum_{n=1}^{q_{K-1}} f(n\alpha + bq_{k-1}\alpha) - q_{K-1}\int_{0}^{1}f(x) \,\mathrm{d}x
\] for every $b \leq b_{K-1}-1$. For $f$ of the form considered in Theorems \ref{ThmUpperDensDiscFunctionsRefined} and \ref{ThmKoksRefined}
and $\frac{b}{a_k} \in [c',c]$ for some $0 \leq c' < c < 1$,
one obtains \[(-1)^{K}\left(S_{(b+1)q_{K-1}}(f,\alpha) - S_{bq_{K-1}}(f,\alpha)\right) \gtrapprox  d\] 
with $d > 0$, provided that
$q_{K-1}$ satisfies some congruence relation that depends on the location of the discontinuities
(Lemma \ref{lem_case_posud}). 
 If $\frac{b}{a_k} \in [0,d')$ we get \[(-1)^{K}\left(S_{(b+1)q_{K-1}}(f,\alpha) - S_{bq_{K-1}}(f,\alpha)\right) \gtrapprox  0.\] 
Thus, for $\delta > 0$ and $N$ with $(\delta + c') a_K < b_{K-1}(N) < c a_K$, we have
\[(-1)^KS_N(f,\alpha) \approx S_{b_{K-1}q_{K-1}}(f,\alpha) \gtrapprox \delta d a_K \gg \sum_{i=1}^{K}a_i.\]
Letting $\delta \to 0$, we see that the desired inequality holds on a proportion of at least
$\frac{c-c'}{c}$ many elements among $\{1,\ldots,ca_K\}$.
By a refinement of the Borel-Bernstein Theorem, we ensure that there are infinitely many odd respectively even $K$ that both satisfy $\sum_{i =1}^{K-1} a_i = o(a_K)$ and this certain congruence relation on $q_{K-1}$. Thus, we obtain the positive upper density by considering the subsequence $ca_K$
where the $K$ are chosen out of this infinite set, giving upper density of at least $\frac{c-c'}{c} > 0$. 

Under the assumptions $\,\sum\limits_{i = 1}^{ \nu} A_i \neq  0 $ or $\;\sum\limits_{i=1}^{ \nu} c_i (x_{i} - x_{i-1}) \neq 0$, it is possible to prove 
that we can choose in the above discussion $c' = 0$ (Lemma \ref{lem_case_ud1}), which leads to the result with upper density $1$.

\subsection{Preparatory Lemmas}
Before we come to the proof of Theorems \ref{ThmUpperDensDiscFunctionsRefined} and \ref{ThmKoksRefined}, we need a few auxiliary results. The first lemma treats the sawtooth function, which in view of Lemma \ref{LemRepresentationf} is a building block for the decomposition of $f$.
\begin{lem}
\label{LemAsymptoticsSN}
Let $\alpha \in [0,1)$ be an irrational number and $f(x) = \left( \{ x\} - \frac{1}{2} \right)$.
If $a_K$ is sufficiently large, then, for $0 \leq b \leq a_K$, we have
\begin{align*}
  S_{bq_{K-1}}( f, \alpha ) & =  (-1)^K\frac{b}{2} \left(1 - \frac{b}{a_{K}}\right) + O \left(1\right).
\end{align*}
\end{lem}

\begin{proof}
We only consider the case where $K$ is odd since the other case can be treated analogously. We thus have 
\begin{equation*}
\frac{1}{q_K + q_{K-1}} \leq \delta_{K-1} = q_{K-1}\alpha - p_{K-1} \leq \frac{1}{q_K }.
\end{equation*}

Since $q_K = a_K q_{K-1} + q_{K-2}$, we get the asymptotic expression $ \delta_{K-1} = \frac{1}{q_{K-1} a_K ( 1 + \varepsilon_K)}$, where $ \varepsilon_K= O( a_K^{-1})$. We obtain
\begin{align*}
S_{b q_{K-1} }( f, \alpha) & = \sum_{n=1}^{b q_{K-1}} \left( \left \{ n \alpha  \right\} - \frac{1}{2}\right) \\
& = \sum_{u=0}^{b-1} \sum_{n=1}^{q_{K-1}} \left( \left \{ n \alpha + u q_{K-1} \alpha   \right\} - \frac{1}{2} \right) \\
& = \sum_{u=0}^{b-1} \sum_{n=1}^{q_{K-1}} \left( \left \{ n \frac{p_{K-1}}{q_{K-1}} + \frac{n}{q_{K-1}} \delta_{K-1} + u \delta_{K-1}   \right\} - \frac{1}{2} \right) \\
& = \sum_{u=0}^{b-1} \sum_{n=1}^{q_{K-1}} \left( \left \{ n \frac{p_{K-1}}{q_{K-1}} + O \left( \frac{1 }{q_{K-1}a_K } \right)  +\frac{u}{q_{K-1}a_K (1 +\varepsilon_K)}   \right\} - \frac{1}{2} \right) \\
& = \sum_{u=0}^{b-1} \sum_{n=0}^{q_{K-1} -1 } \left(   \frac{n}{q_{K-1}} + O \left( \frac{1 }{q_{K-1}a_K } \right)  +\frac{u}{q_{K-1}a_K (1 + \varepsilon_K) }  \right)  - \frac{b q_{K-1} }{2} .
\end{align*}
In the second last line we used that $ \delta_{K-1}= \frac{1}{q_{K-1} a_K ( 1 + \varepsilon_K)}$ and we employed $n \leq q_{K-1}$ in the inner summation. Further, we used that $\gcd(p_{K-1}, q_{K-1} ) =1$ and hence the remainders of $n p_{K-1}$ modulo $q_{K-1}$ are precisely the integers $n=0, \ldots , q_{K-1}-1$. Finally, we omitted the fractional part $ \{ \cdot \} $, since $ \frac{n}{q_{K-1}} + O \left( \frac{1 }{q_{K-1}a_K } \right)  +\frac{u}{q_{K-1}a_K(1+\varepsilon_K)} < 1 $ holds for all $n$ and $u$, provided $a_K$ is sufficiently large. This leads to
\begin{align*}
\sum_{u=0}^{b-1} \sum_{n=0}^{q_{K-1}-1} 
\Bigg( \frac{n}{q_{K-1}}  + O \Bigg( \frac{1}{q_{K-1}a_K}  \Bigg) & +\frac{u}{q_{K-1}a_K(1+\varepsilon_K)} \Bigg) - \frac{b q_{K-1}}{2} \\
&= \frac{q_{K-1}( q_{K-1} -1) b}{2 q_{K-1}} + O(1) + \frac{ b( b -1 ) }{2 a_K (1+\varepsilon_K)} - \frac{b q_{K-1}}{2} \\
&=  - \frac{ b}{2} + \frac{ b^2 }{2 a_K (1+\varepsilon_K)}  + O(1) \\
& = - \frac{b}{2} \left(1 - \frac{b}{a_{K}}\right) + O \left(1\right),
\end{align*}
where we also made use of the estimate $b \leq a_K$.
\end{proof}

The next lemma treats the local discrepancy function, which in view of Lemma \ref{LemRepresentationf} is also a building block for the decomposition of $f$.
\begin{lem}
\label{prop_telescopic_part}
Let $f = \mathds{1}_{\left [ \frac{r_1}{s_1},\frac{r_2}{s_2} \right ]}$ with $r_i,s_i \in \mathbb{N}$ and $\gcd(r_1,s_1) = \gcd(r_2,s_2) = 1$. Further, let $\alpha = [0;a_1,a_2,\ldots]$ be fixed and let both $k \in \N$ and $a_k$ be sufficiently large. Moreover, let $b \in \mathbb{N}$
such that $b  \leq \frac{a_k}{2 s_1 s_2}$. Then, we have

\[ S_{bq_{k-1}}(f,\alpha)= b \left( \left\{\frac{r_1}{s_1}q_{k-1} \right\}
+ (-1)^{k-1}\mathds{1}_{\{s_1 \mid q_{k-1}\}} - \left\{\frac{r_2}{s_2}q_{k-1} \right\}
- (-1)^{k-1}\mathds{1}_{\{s_2 \mid q_{k-1}\}} \right),
\]
where $ \mathbb{1}_{ s_i | q_{k-1} } = 1$ if $s_i | q_{k-1}$ and $ \mathbb{1}_{ s_i | q_{k-1} } = 0$ if $s_i \nmid q_{k-1}$.
\end{lem}
\begin{proof}
We show the statement only in the case where $k$ is odd and $s_1 | q_{k-1}$ and $s_2 \nmid q_{k-1}$. The other cases can be treated analogously. Thus, we need to show that
\begin{equation*}
 S_{bq_{k-1}}(f,\alpha)= b \left( 1 - \left \{ \frac{r_2}{s_2}q_{k-1} \right \} \right).
\end{equation*}
We recall that $S_{b q_{k-1}}(f, \alpha) = \left | \left \{ 1 \leq n \leq b q_{k-1} : \{ n \alpha  \} \in \left [ 
 \frac{r_1}{s_1}, \frac{r_2}{s_2} \right ] \right \} \right | - b q_{k-1} \left(  \frac{r_2}{s_2} -  \frac{r_1}{s_1} \right) $. In the following, we use that for odd $k$, we have $ 0 \leq \delta_{k-1} = {q_{k-1}}\alpha - p_{k-1}  = \frac{1}{a_{k}q_{k-1} (1+o_{a_k}(1))}$. We obtain
 \begin{align*}
\left | \left \{ 1 \leq n \leq b q_{k-1} : \{ n \alpha  \} \in \left [ 
 \frac{r_1}{s_1}, \frac{r_2}{s_2} \right ] \right \} \right |
 & = \sum_{j=0}^{b-1} \left |  \left \{ 1 \leq n \leq q_{k-1} : \{ n \alpha + j q_{k-1} \alpha \} \in \left [ 
 \frac{r_1}{s_1}, \frac{r_2}{s_2} \right ] \right \} \right |  \\
 & = \sum_{j=0}^{b-1} \left | \left \{ 1 \leq n \leq q_{k-1} : \left \{ n \frac{p_{k-1}}{q_{k-1}} + n \frac{\delta_{k-1}}{q_{k-1}} +  j  \delta_{k-1}  \right \} \in \left [ 
 \frac{r_1}{s_1}, \frac{r_2}{s_2} \right ] \right \} \right | \\
& = \sum_{j=0}^{b-1} \left | \left \{ 0 \leq n \leq q_{k-1} - 1 : \left \{  \frac{n}{q_{k-1}} + m(n) \frac{\delta_{k-1}}{q_{k-1}} +  j \delta_{k-1}  \right \} \in \left [ 
 \frac{r_1}{s_1}, \frac{r_2}{s_2} \right ] \right \} \right | ,
 \end{align*}
 where $  m(n) := p_{k-1} n \mod{q_{k-1}}$ and we used that $p_{k-1}$ and $q_{k-1}$ are coprime in the last line. Let $ \varepsilon_n := m(n) \frac{\delta_{k-1}}{q_{k-1}} +  j \delta_{k-1} $. Now we have to count the number of $0 \leq n \leq q_{k-1} -1$ such that $ \frac{n}{q_{k-1}} + \varepsilon_n \in \left [ \frac{r_1}{s_1}, \frac{r_2}{s_2} \right ] $. Since $k$ is odd and we assume that $ b \leq \frac{a_k}{2 s_1 s_2}$, we have $ 0 \leq \varepsilon_n \leq \frac{1}{a_{k} q_{k-1} (1 + o_{a_k}(1))} +  \frac{1}{2 s_1 s_2 q_{k-1} (1 + o_{a_k}(1)) } $. Using $s_1 \mid q_{k-1}$ and $ \varepsilon_n \geq 0$, we see that the smallest integer $n$ with $0 \leq n \leq q_{k-1}-1$ such that $ \frac{n}{q_{k-1}} + \varepsilon_n \geq \frac{r_1}{s_1} $ is $n = \frac{r_1}{s_1}q_{k-1}$. Since $s_2 \nmid q_{k-1}$ and $ \varepsilon_n \leq \frac{1}{a_{k} q_{k-1} (1 + o_{a_k}(1))} +  \frac{1}{2 s_1 s_2q_{k-1} (1 + o_k(1)) } < \frac{1}{s_2 q_{k-1}}$ (since $a_k \geq  2 s_1 s_2$, if $a_k$ is sufficiently large, where also $1 +o_{a_k}(1)$ is close to $1$), we have that the largest integer $n$ with $0 \leq n \leq q_{k-1} -1$ such that $ \frac{n}{q_{k-1}} + \varepsilon_n \leq \frac{s_2}{r_2}$ is $n = \left \lfloor \frac{r_2}{s_2} q_{k-1} \right \rfloor$. This means, the number of integers $0 \leq n \leq q_{k-1} -1$ such that  $  \frac{n}{q_{k-1}} + \varepsilon_n   \in \left [ 
 \frac{r_1}{s_1}, \frac{r_2}{s_2} \right ] $ is equal to $ \left \lfloor \frac{r_2}{s_2} q_{k-1} \right \rfloor - \frac{r_1}{s_1} q_{k-1} +1 = 1 - \left \{ \frac{s_2}{r_2} q_{k-1} \right \} + q_{k-1} \left ( \frac{r_2}{s_2} - \frac{r_1}{s_1}\right)$. This leads to
 \begin{align*}
\sum_{j=0}^{b-1} \left | \left \{ 1 \leq n \leq q_{k-1}: \left \{ n \frac{p_{k-1}}{q_{k-1}} + n\frac{\delta_{k-1}}{q_{k-1}} +  j \delta_{k-1} \right \} \in \left [ 
 \frac{r_1}{s_1}, \frac{r_2}{s_2}  \right ] \right \} \right |  = b \left( 1 - \left \{ \frac{s_2}{r_2} q_{k-1} \right \} + q_{k-1} \left ( \frac{r_2}{s_2} - \frac{r_1}{s_1} \right) \right).
 \end{align*}
 This implies that $ S_{b q_{k-1} } ( f, \alpha) = b \left( 1 -  \left \{ \frac{s_2}{r_2} q_{k-1} \right \} \right)$, as claimed.
\end{proof}

\begin{lem}
\label{lem_case_ud1}
Let $f$ be as in Definition \ref{def_fct1}, and let $A_i,c_i,g$ as in Lemma \ref{LemRepresentationf}.
Further, assume that $ \sum_{i = 1}^{ \nu} A_i \neq  0 $ or $ \sum_{i=1}^{ \nu} c_i ( x_{i} - x_{i-1}) \neq 0$.
Then, there exist constants $c, c' > 0$, and integers $\alpha_j,\beta_j,\gamma_j,\delta_j, j = 1,2$ (depending on $f$) such that the following holds:
\begin{itemize}
\item If $K \equiv \alpha_1 \pmod{\beta_1}, \quad q_{K-1} \equiv \gamma_1 \pmod{\delta_1}$, then for any integer $b$ with $0 \leq b \leq ca_{K}$, we have for sufficiently large $K$ and $a_K$,
\begin{equation}
\label{EqPosEstimate}
S_{bq_{K-1}}(g,\alpha) \geq b c' + O \left( \sum_{i=1}^{K-1} a_i \right).
\end{equation}
\item If $K \equiv \alpha_2 \pmod{\beta_2}, \quad q_{K-1} \equiv \gamma_2 \pmod{\delta_2}$, then for any integer $b$ with $0 \leq b \leq c a_{K}$, we have have for sufficiently large $K$ and $a_K$,
\begin{equation}
\label{EqNegEstimate}
S_{bq_{K-1}}(g,\alpha) \leq -bc' + O \left( \sum_{i=1}^{K-1} a_i \right).
\end{equation}
\end{itemize}
\end{lem}

\begin{proof}
We only prove \eqref{EqPosEstimate}, the inequality \eqref{EqNegEstimate} can be shown analogously.
First, assume that $ \sum_{i=1}^{\nu} A_i > 0$. Let $x_i = \frac{r_i}{s_i}$ with $\gcd(r_i, s_i)=1$. We set $c :=  \frac{1}{4  s_1 \cdots s_{\nu} }$ and choose $\alpha_1 := 1, \beta_1 :=2, \gamma_1:= 0$ and $ \delta_1 := s_1\cdots s_{\nu}$. Assume that $K \in \N$ satisfies $ K \equiv \alpha_1 \Mod{\beta_1}$ and $q_{K-1} \equiv \gamma_1 \Mod{\delta_1}$.
\par{}
By Lemma \ref{LemRepresentationf}, we can write $g(x) = \left( \sum_{i=1}^{\nu} A_i \right) \left( \left \{ x \right \} - \frac{1}{2} \right) + \sum_{i=1}^{\nu} c_i \left( \mathbb{1}_{[x_{i-1}, x_i]}(\{x\} ) - (x_{i} - x_{i-1} ) \right)$, where we set $x_0 := 0$. This leads to 
\begin{equation*}
S_{bq_{K-1}}(g, \alpha) = \left( \sum_{i=1}^{\nu} A_i \right) \sum_{n=1}^{b q_{K-1}} \left( \{ n \alpha \} - \frac{1}{2} \right) + \sum_{n=1}^{b q_{K-1}} \sum_{i=1}^{\nu} c_i \left( \mathbb{1}_{[x_{i-1}, x_i]}( \{n \alpha \} ) - (x_{i} - x_{i-1} ) \right).
\end{equation*}
By Lemma \ref{prop_telescopic_part}, the second sum above is equal to $0$. Since $K-1$ is even, we get by Lemma \ref{LemAsymptoticsSN}
\begin{align*}
\left( \sum_{i=1}^{\nu} A_i \right) \sum_{n=1}^{b q_{K-1}} \left( \{ n \alpha \} - \frac{1}{2} \right) & \geq  \left( \sum_{i=1}^{\nu} A_i \right) \frac{b}{2} \left( 1 - \frac{b}{a_K}  \right)  + O \left( 1 \right) \\
& \geq \left( \sum_{i=1}^{\nu} A_i \right) \frac{b}{2}  \left( 1 -  c\right) + O \left( 1 \right).
\end{align*}
We define $c':= \frac{1}{2} \left( \sum_{i=1}^{\nu} A_i \right)( 1 -c) $ which is a positive constant, since $c < 1$. This finishes the proof in case of $ \sum_{i=1}^{\nu} A_i > 0$. The case $ \sum_{i=1}^{\nu} A_i < 0 $ can be handled analogously.
\par{}
Now assume $ \sum_{i=1}^{\nu} A_i = 0$, but $ \sum_{i=1}^{\nu} c_i ( x_i - x_{i-1}) \neq 0$. This implies $ \nu \geq 2$ and we assume without loss of generality $x_1 > 0$. Further let $ \sum_{i=1}^{\nu} c_i ( x_i - x_{i-1}) > 0$, the case where $  \sum_{i=1}^{\nu} c_i ( x_i - x_{i-1}) < 0$ can be treated analogously. Let $x_i = \frac{r_i}{s_i}$ with $\gcd(r_i,s_i)=1$ for $i=0, \ldots, \nu$. We choose $ \alpha_1 = 0$, $ \beta_1 = 2$, $ \gamma_1 = s_1 \cdots s_{\nu}-1$ and $ \delta_1 = s_1 \cdots s_{\nu}$ (we note that $s_1 \cdots s_{\nu}  \geq  2$). Let $K \in \N$ such that $ K \equiv \alpha_1 \Mod{\beta_1}$, $q_{K-1} \equiv \gamma_1 \Mod{\delta_1}$ and define $c:= \frac{1}{4 s_1 \cdots s_{\nu}}$. Then, for $0 \leq b \leq c a_K$, we get
\begin{align*}
S_{b q_{K-1}} & = \sum_{n=1}^{bq_{K-1}} \sum_{i=1}^{\nu} c_i \left( \mathbb{1}_{[x_{i-1}, x_i]}( \{n \alpha \} )- (x_i - x_{i-1}) \right) \\ 
& = b \sum_{i=2}^{\nu} c_i   \left(\left \{ \frac{r_{i-1}}{s_{i-1}} q_{K-1} \right\} -  \left \{ \frac{r_i}{s_i} q_{K-1} \right\}  \right) + b c_1 \left( 1 - \left \{ \frac{r_1}{s_1} q_{K-1} \right \} \right).
\end{align*}
In the last line, we applied Lemma \ref{prop_telescopic_part} and used that $ s_i \nmid q_{K-1}$ for every $i=1, \ldots , \nu $. By the choice of $ \gamma_1$ and $\delta_1$, we have $q_{K-1} \equiv -1 \Mod{s_i}$ for all $i=1, \ldots, \nu$ and therefore 
\begin{equation*}
b \sum_{i=2}^{\nu} c_i   \left(\left \{ \frac{r_{i-1}}{s_{i-1}} q_{K-1} \right\} -  \left \{ \frac{r_i}{s_i} q_{K-1} \right\}  \right) + b c_1 \left( 1 - \left \{ \frac{r_1}{s_1} q_{K-1} \right \} \right) = 
b \sum_{i=1}^{\nu} c_i   \left( \frac{r_i}{s_i}  -  \frac{r_{i-1}}{s_{i-1}}  \right).
\end{equation*}
By now defining $c' =  \sum_{i=1}^{\nu} c_i   \left( \frac{r_i}{s_i}  -  \frac{r_{i-1}}{s_{i-1}}  \right) =  \sum_{i=1}^{\nu} c_i  (x_i - x_{i-1}) > 0$, the proof is finished.
\end{proof}

Next, we consider the analogue of the previous lemma in the case $ \sum_{i = 1}^{ \nu} A_i = \sum_{i=1}^{ \nu} c_i ( x_{i} - x_{i-1}) =0$, where we aim for a positive upper density.

\begin{lem}
\label{lem_case_posud}
Let $f$ be as in Definition \ref{def_fct1}, and let $A_i,c_i, g$ as in Lemma \ref{LemRepresentationf}. Further, assume that $ \sum_{i = 1}^{ \nu} A_i = \sum_{i=1}^{ \nu} c_i ( x_{i} - x_{i-1}) = 0$.
Then, there exist constants $c,c', d > 0$ with $ c' < c $ and integers $\alpha_j,\beta_j,\gamma_j,\delta_j, j = 1,2$ (all depending on $f$) such that the following holds:
\begin{itemize}
\item
If $K \equiv \alpha_1 \pmod{\beta_1}, \quad q_{K-1} \equiv \gamma_1 \pmod{\delta_1}$, then for $ c' a_{K} \leq b \leq  c a_{K}$, we have for sufficiently large $K$ and $a_K$ 
\begin{equation}
\label{Eq_lowerbound_Lem_posud}
S_{bq_{K-1}}(g,\alpha) \geq d a_K .
\end{equation}
\item
If $K \equiv \alpha_2 \pmod{\beta_2}, \quad q_{K-1} \equiv \gamma_2 \pmod{\delta_2}$, then for $ c'a_{K} \leq b \leq c a_{K}$, we have for sufficiently large $K$ and $a_K$ 
\begin{equation}
\label{Eq_upperbound_Lem_posud}
S_{bq_{K-1}}(g,\alpha) \leq -d a_K.
\end{equation}
\end{itemize}
\end{lem}

\begin{proof}
We only show \eqref{Eq_lowerbound_Lem_posud}, since \eqref{Eq_upperbound_Lem_posud} can be shown analogously. By assumption $ c_1 = \sum_{i=1}^{\nu} A_i = 0$, thus there exist at least two $A_i \neq 0$. To keep notation simple we assume $A_1 \neq  0$ and thus $ c_2 = \sum_{i=2}^{\nu}A_i \neq 0 $. We assume $ c_2 > 0$ and note that the case where $c_2 < 0$ can be handled similarly.
\par{}
Let $x_i = \frac{r_i}{s_i}$ for all $i=0, \ldots , \nu $. We choose $ \alpha_1 = 1$, $ \beta_1 = 2$, $ \gamma_1 = 1$ and $ \delta_1 =  s_1 \cdots s_{\nu}$. We take $K \in \N$ with $K \equiv \alpha_1 \Mod{\beta_1}$, $ q_{K-1} \equiv \gamma_1 \Mod{\delta_1}$ and $a_K > \frac{8}{(x_2 - x_1) c_2} $. Moreover, let $c' = \frac{x_1 + x_2}{2}$ and $c= x_2$, which implies $0 < c' < c$. Let $b \in \N$ with $c' a_K < b < c a_K$ and consider 
\begin{align*}
S_{b q_{K-1}}( g, \alpha ) & = \sum_{n=1}^{b q_{K-1}} \sum_{i=1}^{\nu} c_j \left ( \mathbb{1}_{[x_{i-1}, x_i]}( \{ n \alpha \} ) - (x_i -x_{i-1}) \right) \\
& = \sum_{i=1}^{\nu}   c_i \left ( \left | \left \{ 1 \leq n \leq b q_{K-1} : \{ n \alpha \} \in [x_{i-1}, x_i] \right\} \right | - b q_{K-1} (x_i -x_{i-1}) \right).
\end{align*}
Now, for every $1 \leq i \leq \nu$, we provide a suitable estimate for $ \left |  \left \{ 1 \leq n \leq b q_{K-1} : \{ n \alpha \} \in [x_{i-1}, x_i] \right \} \right | -  b q_{K-1} (x_i -x_{i-1}) $ from below. Starting with the case $i \geq 3 $, we observe that
\begin{align*}
\left |  \left \{ 1 \leq n \leq b q_{K -1 } : \{ n \alpha \} \in [x_{i-1},x_{i}] \right \} \right |  = \sum_{u=0}^{b -1 }  \left | \left \{ 1 \leq n \leq q_{K -1 } : \{ n \alpha + u q_{K -1 } \alpha \} \in [x_{i-1},x_{i}] \right \} \right |.
\end{align*}

For $0 \leq u \leq b -1 $, we have $ \left | \left \{ 1 \leq n \leq q_{K-1} : \{ n \alpha + u q_{K-1} \alpha \} \in [x_{i-1},x_{i}] \right \} \right | = \left | \left \{ 1 \leq n \leq q_{K-1} : \left \{ n \frac{p_{K-1}}{q_{K-1}} + \varepsilon_n \right  \} \in [x_{i-1},x_{i}] \right \} \right | $, where $  \varepsilon_n := u q_{K-1} \left( \alpha - \frac{p_{K-1}}{q_{K-1}} \right) + n \left( \alpha - \frac{p_{K-1}}{q_{K-1}} \right) $. Since $2 | (K-1)$ and $ u < b \leq x_2 a_{K} $, we have that $ 0 \leq \varepsilon_n  \leq \frac{x_2}{q_{K-1}}$. Thus, we get 
\begin{align*}
\left | \left \{ 1 \leq n \leq q_{K-1} :\left \{ n \frac{p_{K-1}}{q_{K-1}} + \varepsilon_n  \right \} \in [x_{i-1},x_{i}] \right \} \right | 
& \geq 
\left | \left \{ 0 \leq n \leq q_{K-1} -1 :  \frac{n}{q_{K-1}}  \geq x_{i-1} ,  \frac{n}{q_{K-1}} + \frac{x_2}{q_{K-1}}  \leq x_{i} \right \} \right | \\
& = \lfloor x_i q_{K-1} \rfloor - \lceil x_{i-1} q_{K-1} \rceil + 1, \\
\end{align*}
where we used that, for $i \geq 3$, we have $x_i > x_2$ and thus, the largest integer $n$ such that $\frac{n}{q_{K-1}} + \frac{x_2}{q_{K-1}} \leq x_i$ is $ n= \lfloor x_i q_{K-1} \rfloor $. Moreover, the smallest $n$ such that $ \frac{n}{q_{K-1}} \geq x_{i-1}$ is $n= \lceil x_{i-1} q_{K-1} \rceil $. Since $q_{K-1} \equiv 1 \Mod{s_1 \ldots s_{\nu}}$, we have $ \{ q_{K-1} x_i \} = x_i$ for any $1 \leq i \leq \nu$ (recall that the $s_i$ denote the denominators of the rationals $x_i$) and thus, in case of $ i \geq 3$, we have shown that 
\begin{equation*}
\left | \left \{ 1 \leq n \leq q_{K-1} : \{ n \alpha + u q_{K-1} \alpha \} \in [x_{i-1},x_{i}] \right \} \right | \geq (x_i -x_{i-1}) q_{K-1} - (x_i -x_{i-1})
\end{equation*}
holds for any $0 \leq u \leq b - 1$. For the case $i=1$, a similar argument as before shows that for any $0 \leq u \leq b -1$, we have
\begin{equation*}
 \left | \left \{ 1 \leq n \leq q_{K-1} : \{ n \alpha + u q_{K-1} \alpha \} \in [0,x_1] \right \} \right | \geq x_1 q_{K-1} - x_1.
\end{equation*}
 
 Now we turn our attention to the case of $i=2$, where we establish a slightly different lower bound for $  \left | \left \{ 1 \leq n \leq q_{K-1} : \{ n \alpha + u q_{K-1} \alpha \} \in [x_1, x_2] \right \} \right | $. First, we consider those $u$ with $0 \leq u < x_1 a_{K} (1 + \varepsilon)$. We choose $ \varepsilon > 0$ small enough such that $x_1 (1 + 4 \varepsilon ) < x_2 $. In that case, we get the same lower bound as before, i.e., we establish $  \left | \left \{ 1 \leq n \leq q_{K-1} : \{ n \alpha + u q_{K-1} \alpha \} \in [x_1, x_2] \right \} \right | \geq (x_2 -x_1) q_{K-1} - (x_2 -x_1)$. We are now in the case, where $ x_1 a_{K} (1 + \varepsilon)  \leq u \leq x_2 a_{K} -1 $ (which is a complete case distinction, since $u < b \leq x_2 a_{K}$), we get $ \left | \left \{ 1 \leq n \leq q_{K-1} : \left \{ n \alpha + u q_{K-1} \alpha  \right \} \in [x_1, x_2] \right \} \right | = \left | \left \{ 1 \leq n \leq q_{K-1} : \left \{ n \frac{p_{K-1}}{q_{K-1}} + \varepsilon_n \right \} \in [x_1, x_2] \right \} \right | $, where we again set $ \varepsilon_n = u q_{K-1} \left( \alpha - \frac{p_{K-1}}{q_{K-1}} \right) + n \left( \alpha - \frac{p_{K-1}}{q_{K-1}} \right)  $. Since $2 | (K-1)$ and $ a_{K} x_1 ( 1 + \varepsilon ) \leq u \leq  a_{K} x_2 - 1 $, we have that for any $1 \leq n \leq q_{K-1}$, $ \frac{x_1 (1+ \varepsilon)}{q_{K-1}( 1 + o_{a_K}(1))} \leq \varepsilon_n \leq \frac{x_2 }{q_{K-1}} $. Further, we use that $ \frac{1+ \varepsilon}{1 + o_{a_K}(1)} \geq 1$ if $K$ is sufficiently large and hence $ \frac{x_1 }{q_{K-1}} \leq \varepsilon_n$. This gives us the estimate
\begin{align*}
\left | \left \{ 1 \leq n \leq q_{K-1} : \left \{ n \frac{p_{K-1}}{q_{K-1}} + \varepsilon_n \right \} \in [x_1, x_2] \right \} \right |
& \geq \left | \left \{ 0 \leq n \leq q_{K-1} -1 :   \frac{n}{q_{K-1}} + \frac{x_1 }{q_{K-1}} \geq x_1,   \frac{n}{q_{K-1}} + \frac{x_2}{q_{K-1}} \leq x_2 \right \} \right | .
\end{align*}
Here we used $\gcd(p_{K-1}, q_{K-1})=1$ which implies that $ n p_{K-1}$ runs through all remainder classes modulo $q_{K-1}$.
The smallest $0 \leq n \leq q_{K-1} -1$ such that $  \frac{n}{q_{K-1}} + \frac{x_1}{q_{K-1}} \geq x_1 $ is $n = \lfloor x_1 q_{K-1} \rfloor = x_1 q_{K-1} - x_1$, which follows from the congruence relation satisfied by $ q_{K-1}$. The largest $n$ such that $ \frac{n}{q_{K-1}} + \frac{x_2}{q_{K-1}} \leq x_2$ is $n= \lfloor x_2 q_{K-1} \rfloor = x_2 q_{K-1} - x_2$. These estimates lead to
\begin{align*}
\left | \left \{ 0 \leq n \leq q_{K-1} -1 :   \frac{n}{q_{K-1}} + \frac{x_1 }{q_{K-1}} \geq x_1,   \frac{n}{q_{K-1}} + \frac{x_2}{q_{K-1}} \leq x_2 \right \} \right | & = \lfloor x_2 q_{K-1} \rfloor - \lfloor x_1 q_{K-1} \rfloor + 1 \\
& = (x_2 -x_1) q_{K-1} + (1 -(x_2 -x_1)) . 
\end{align*}
Now we can combine all the estimates we obtained before, in order to get
\begin{align*}
\sum_{i=1}^{\nu} c_i \sum_{u=0}^{b -1 }  \left | \left \{ 1 \leq n \leq q_{K-1} : \{ n \alpha + u q_{K-1} \alpha \}   \in [x_{i-1}, x_i] \right \} \right | & - (x_i - x_{i-1}) q_{K-1} \\
& \geq - \sum_{i=1}^{\nu} c_i \sum_{u=0}^{b -1 }   (x_i - x_{i-1})   \\ 
& \phantom{=}+  c_2 \left | \{ 0 \leq u \leq b- 1 :   a_{K} x_1 ( 1 + \varepsilon ) \leq u \leq  a_{K} x_2 - 1 \} \right | \\
& = 0 + c_2 (  b - \lceil a_{K} x_1 (1 + \varepsilon ) \rceil ) \\
& \geq a_{K} c_2 \left ( \frac{x_1 + x_2}{2} - (1 + \varepsilon ) x_1  \right)  - 2 \\
& \geq a_{K} c_2 \left ( \frac{x_1 + x_2}{2} - (1 + 2 \varepsilon ) x_1  \right) . 
\end{align*}
We used the overall assumption of $ \sum_{i=1}^{\nu} c_i  (x_i - x_{i-1}) = 0$ and $ -2 \geq - c_2 a_K \varepsilon  $, since $a_K$ is large. The proof is now finished by defining $d := c_2 \left ( \frac{x_1 + x_2}{2} - (1 + 2 \varepsilon )  x_1 \right)  > 0$. 
\end{proof}

\subsection{Refining the Borel-Bernstein Theorem}

We see that Lemma \ref{lem_case_ud1} and Lemma \ref{lem_case_posud} give us lower bounds on Birkhoff sums, provided that $a_K$ does not only dominate the sum of the preceding partial quotients, but both
$K$ and $q_{K-1}$ additionally satisfy certain modularity conditions. Without having to satisfy these extra conditions, the existence of infinitely many such $a_K$ for generic $\alpha$ could be deduced from a combination of the Borel-Bernstein Theorem and the estimate \eqref{Eq_trimmed_sum} on the trimmed sum of partial quotients. The aim of this section is to establish a version of the Borel-Bernstein Theorem (Lemma \ref{LemInfiniteSets}) that allows to include additional assumptions on $K$ and $q_{K-1}$. We make use of some known auxiliary results that are stated for the reader's convenience in full detail below.

\begin{lem}(Lemma C in \cite{ErdoesRenyi1959})
\label{LemBorelCantelli}
Let $(\Omega, \mathcal{F}, \mathbb{P})$ be a probability space with events $(A_n)_{n \in \N}$ such that $ \sum_{n=1}^{\infty} \mathbb{P} [A_n] = \infty $ and
\begin{equation*}
\limsup_{N \rightarrow \infty } \frac{ \left(  \sum_{K=1}^{N} \mathbb{P}[A_K] \right)^2 }{\sum_{K,L=1}^{N} \mathbb{P}[A_K \cap A_L]} =1.
\end{equation*}
Then, 
\begin{equation*}
\mathbb{P}[ \limsup_{n \rightarrow \infty} A_n] = 1.
\end{equation*}
\end{lem}

\begin{lem}(Lemma 2.5 in \cite{Kesten1960})
\label{LemKesten} 
Let $v \in \N$ with $v \geq 2$, and $0 \leq u \leq v-1$ and define $k = k(v) := \left | \{(u_1,u_2): 0 \leq u_1,u_2 < v: \gcd(u_1,u_2,v) = 1\} \right | $. Further, let $m,p \in \N$, $ \alpha  \sim \U ( [0,1))$, define \\ $\mathscr{F} := \sigma \big( a_{1}(\alpha ),..., a_{m}(\alpha ) \big) $ and let $E \in \sigma \big( a_{m+p+1}(\alpha ), a_{m+p+2}(\alpha ),... \big)$. Then, there exist constants $C > 0$ and $ \lambda \in (0,1)$ such that
\begin{equation}\label{decay} 
\left | \mathbb{P} \left [  q_{m+p}(\alpha ) \equiv u \Mod{v}, \alpha  \in E \ | \mathscr{F} \right]  - \frac{1}{k} \mathbb{P} \left[\alpha \in E \right] \right | \leq C \lambda^{\sqrt{p}} .
\end{equation}
\end{lem}

\begin{rem}
    In \cite{Kesten1960}, the quantity $k$ from Lemma \ref{LemKesten} is not given explicitly. The fact that all pairs $(u_1,u_2)$ with $ \gcd (u_1,u_2,v) = 1$ are admissible and thus $k(v)$ can be defined as in Lemma \ref{LemKesten} follows from \cite{Kesten1962}. The decay rate of \eqref{decay} was improved in \cite{Szusz} to be exponential, that is of the form $C\lambda^{p}$.
\end{rem}

\begin{lem}
\label{LemInfiniteSets}
Let $\psi : \R_+ \rightarrow \R_+ $ be a monotonically increasing function such that $ \sum_{K=1}^{\infty} \frac{1}{\psi(K)} = \infty$ and let $b,d \in \N$ with $b,d \geq 2$. Then, for any fixed $0 \leq a \leq b-1$ and $0 \leq c \leq d-1$, for almost all $\alpha \in [0,1)$, the set
\begin{align*}
&  \left \{ K \in \N \ \Bigg | \  K \equiv a \Mod{b} , q_{K-1}  \equiv  c \Mod{d}, \psi(K) < a_K < K^2, \sum_{i=1}^{K-1}a_i \leq 2  K \log K  \right \}
\end{align*}
has infinite cardinality.
\end{lem}

\begin{rem}
    In particular, Bernstein's Theorem can be strengthened in the following way:
 For any monotonic increasing function $\psi: \mathbb{\R_+ } \to \mathbb{R}_{+}$ and any positive
 integers $a,b,c,d$ we have, for almost every $\alpha \in [0,1)$,
    \begin{equation*}
        \left | \left\{K \in \mathbb{N}: a_K > \psi(K),  K \equiv a \Mod{b} , q_{K-1}  \equiv  c \Mod{d}\right\} \right |  \text{ is } \begin{cases} \text{ infinite \hspace{5mm}if\;\; } \sum_{K = 0}^{\infty} \frac{1}{\psi(K)} = \infty,
        \\ \text{ finite \hspace{7mm} if \;\;}\sum_{K = 0}^{\infty} \frac{1}{\psi(K)} < \infty.\end{cases}
    \end{equation*}

The method of proof even allows replacing the condition $ K \equiv a \Mod{b}$ with a condition of the form $K \in A \subseteq \N$, where $A$ has positive lower density. For our purpose the given version is sufficient.
\end{rem}

\begin{proof}
We first show that for almost every $ \alpha \in [0,1)$, the set $ \big  \{ K \in \N \ | \ K \equiv a \Mod{b} , \psi(K) < a_K < K^2 \big \} $ has infinite cardinality. We can assume without loss of generality that $\liminf_{K \to \infty} \frac{\psi(K)}{K \log K} = \infty$ since the result then also follows for all slower growing $\psi $. Now we define
\begin{equation*}
    \tilde{\psi }(K) := 
    \begin{cases}
     \psi \left( K \right ) & \text{ if } K \equiv a \pmod b\\
     K^{2} & \text{ else}.
    \end{cases}
\end{equation*}

Since $\psi$ is monotone, we have that $\sum_{K=1}^{\infty} \frac{1}{\tilde{\psi}(K)} = \infty$ 
$\left( \text{since } \sum_{K =1}^{Nb} \frac{1}{\tilde{\psi}(K)} \geq \frac{1}{b}\sum_{K =b}^{Nb} \frac{1}{\psi(K)} \underset{N \to \infty}{\to} \infty\right).$
By \eqref{Eqbernstein}, there exist infinitely many $K$ such that $a_K > \tilde{\psi}(K)$. Again by \eqref{Eqbernstein}, there are only finitely many $K \in \N$ such that $a_K > K^2$ and thus, we can conclude that there are infinitely many $K \in \N $ with
$K \equiv a \pmod b$ and $\psi(K) < a_K < K^2$. Now we introduce the sets 
\begin{align*}
E_K &:= \left \{ \alpha \in [0,1) \ | \  K \equiv a \Mod{b} , \psi(K) < a_K < K^2 \right  \}, \\
A_K &:= \left  \{ \alpha \in [0,1) \ | \   K \equiv a \Mod{b} , q_{K-1} \equiv c \Mod{d}, \psi(K) < a_K < K^2 \right \}.
\end{align*}
In the following, we show that Lemma \ref{LemBorelCantelli} can be applied to the sequence of sets $(A_{K})_{K \in \N}$. To that end, we define $ k = \left | \{ 0 \leq u_1,u_2 \leq d-1 \ | \ \gcd(u_1,u_2,d)=1 \} \right | $ and we note that $E_K$ only depends on $a_{K}$. Further, we will denote the $1$-dimensional Lebesgue measure on $[0,1)$ by $\mathbb{P}$. Using Lemma \ref{LemKesten}, we get
\begin{equation}
\label{EqProbAKviaEK}
\mathbb{P} \left[ A_K \right] = \frac{1}{k} \mathbb{P} \left[ E_K \right] + O \left( \lambda^{\sqrt{K}} \right),
\end{equation}
with $\lambda \in (0,1)$. This gives us $\sum_{K=1}^{ \infty} \mathbb{P} \left[ A_K \right] = \infty$, since there exist infinitely many $K \in \N$ such that $ \mathbb{P}[E_K] =1$ by the first part of this proof. This shows the first assumption in Lemma \ref{LemBorelCantelli}, i.e. $\sum_{K = 1}^{\infty } \mathbb{P}[A_K] = \infty $. Now we take $K,L \in \N$ with $L+1 \leq K$ and consider 
\begin{align*}
\mathbb{P}\left[ A_K \cap A_L \right] & = \mathbb{P}[A_L] \mathbb{P} \left[ A_K | A_L \right] \\
&= \mathbb{P}\left[ A_L \right] \mathbb{P} \left[ \psi(K) <  a_{K} < K^2  , q_{K-1} \equiv c \Mod{d} , K \equiv a \Mod{b} | A_L \right] \\
&= \mathbb{P}\left[ A_L \right] \left(  \frac{1}{k} \mathbb{P}\left[ E_K \right] + O\left( \lambda^{\sqrt{K-L}} \right) \right) \\
& = \mathbb{P}\left[ A_L \right] \mathbb{P} \left[ A_K \right] + \mathbb{P}[A_L] O\left( \lambda^{\sqrt{K-L}} \right),
\end{align*}
where we employed Lemma \ref{LemKesten} and used the estimate from \eqref{EqProbAKviaEK} in the last line. We fix $N \in \N$ sufficiently large and consider
\begin{equation}
\label{EqSumofProbsAK}
    \sum_{K,L =1}^N \mathbb{P} \left[ A_K \cap A_L \right] =  \sum_{K,L =1}^N \mathbb{P} \left[ A_K \right]  \mathbb{P} \left[ A_L \right] + 2 \sum_{L+1 \leq K \leq N} \mathbb{P}[A_L] O \left( \lambda^{ \sqrt{K-L}} \right) +  \sum_{K=1}^{N} \left( \mathbb{P} \left[ A_K \right] - \mathbb{P} \left[ A_{K} \right]^2 \right) .
\end{equation}
Next, we obtain an upper bound for two of the sums in \eqref{EqSumofProbsAK}. First, we get
\begin{equation*}
2 \sum_{L+1 \leq K \leq N} \mathbb{P}[A_L] O \left( \lambda^{ \sqrt{K-L}} \right) = 2 \sum_{L=1}^{N-1} \mathbb{P} \left[ A_L \right] \sum_{K=L+1}^N O \left( \lambda^{\sqrt{K-L}} \right) \leq O(1) \sum_{L=1}^N\mathbb{P} \left[ A_L \right],
\end{equation*}
where we used that $\lambda \in (0,1)$. Moreover, we get
\begin{equation*}
 \sum_{K=1}^{N} \left( \mathbb{P} \left[ A_K \right] - \mathbb{P} \left[ A_{K} \right]^2 \right) \leq \sum_{K=1}^N \mathbb{P}[A_K]
\end{equation*}
and thus, we have
\begin{align*}
 \sum_{L+1 \leq K \leq N} \mathbb{P}[A_L] O \left( \lambda^{ \sqrt{K-L}} \right) + \sum_{K=1}^{N} \left( \mathbb{P} \left[ A_K \right] - \mathbb{P} \left[ A_{K} \right]^2 \right) \leq O(1) \sum_{K=1}^N \mathbb{P}[A_K] = o \left( \left( \sum_{K=1}^N \mathbb{P}[A_K] \right)^2 \right).
\end{align*}
For the equality in the previous equation, we used that $ \sum_{K=1}^{ \infty } \mathbb{P} [A_K] = \infty$. In total, we have shown that 
\begin{equation*}
\limsup_{N \rightarrow \infty } \frac{ \left(  \sum_{K=1}^{N} \mathbb{P}[A_K] \right)^2 }{\sum_{K,L=1}^{N} \mathbb{P}[A_K \cap A_L]} = \limsup_{N \rightarrow \infty } \frac{1}{1+ o_N(1)} =1.
\end{equation*}
Lemma \ref{LemBorelCantelli} now gives us that $\mathbb{P} \left[ \limsup_{K \rightarrow \infty} A_K \right] =1 $ or, in other words, for almost all $\alpha \in [0,1)$, there are infinitely many $K \in \N$ such that $ \psi(K) < a_{K} < K^2 $ , $K \equiv a \Mod{b}$ and  $q_{K-1} \equiv c \Mod{d}$. 
\par{}
For $K$ sufficiently large, we have
$\psi(K) > \frac{K \log K}{\log (2)}$, so \eqref{Eq_trimmed_sum} shows that
$\max\limits_{\ell \leq K-1} a_{\ell} < a_K$ and thus, applying \eqref{Eq_trimmed_sum} again
leads to $\sum\limits_{i= 1}^{K-1} a_{i} \leq 2 K \log K$ (note that $ \frac{1}{\log 2} \leq 2$).
\end{proof}

We have now all ingredients to turn our attention to the proofs of Theorem \ref{ThmUpperDensDiscFunctionsRefined} and Theorem \ref{ThmKoksRefined}.

\subsection{Proofs of Theorem \ref{ThmUpperDensDiscFunctionsRefined} and Theorem \ref{ThmKoksRefined}}

Note that the class of functions considered in both Theorems \ref{ThmUpperDensDiscFunctionsRefined} and \ref{ThmKoksRefined} is closed under translation by rational numbers, and by Lemma \ref{LemRepresentationf} the same holds for the condition $\sum_{i =1}^{\nu}A_i \neq 0 \text{ or } \sum_{i =1}^{\nu}c_i(x_{i}-x_{i-1}) \neq 0$. Thus, we can assume without loss of generality that $q = 0$. Moreover, we can assume that $ \lim_{K \rightarrow \infty } \frac{\psi(K)}{K \log K} = \infty $ since the result then follows also for slower growing $\psi$.

Assume first that $\sum_{i = 1}^{ \nu} A_i \neq  0 $ or $ \sum_{i=1}^{ \nu} c_i ( x_{i} - x_{i-1}) \neq 0$.
Let $c,c', \alpha_1,\beta_1,\gamma_1,\delta_1$ be as in Lemma \ref{lem_case_ud1}. By Lemma \ref{LemInfiniteSets}, for almost every $\alpha$, there exist infinitely many $K$ such that
\begin{align*}
&  \left \{ K \in \N \Bigg |  K \equiv \alpha_1 \Mod{\beta_1} , q_{K-1}  \equiv  \gamma_1 \Mod{\delta_1}, \tilde{\psi}(K) < a_K < K^2, \sum_{i=1}^{K-1}a_i \leq 2  K \log K  \right \},
\end{align*}
where $\tilde{\psi}(k) = C_1 \psi( C_2 k)$ with $C_1, C_2 > 0$ specified later. Denote by $(k_j)_{j \in \mathbb{N}}$ the increasing sequence of integers such that the above holds.
Now let $N \leq c a_{k_j}q_{k_j-1}$ be arbitrary. Thus, we can write $N = b_{k_j-1}q_{k_j-1} + N'$ where $N' < q_{k_j -1}$ and  $b_{k_j-1} \leq c a_{k_j}$. Defining $g$ as in Lemma \ref{LemRepresentationf} and $\Tilde{g}(x) = g(x + b_{k_j-1}q_{k_j-1}\alpha)$, we have

\begin{equation}
\label{pf_ud_1}
\begin{split}
S_N(f,\alpha) &= S_N(g,\alpha) + O(1)\\
&= S_{b_{k_j-1}q_{k_j-1}}(g,\alpha) +  S_{N'}(\Tilde{g},\alpha) + O(1)\\
&\geq   b_{k_j-1}c' +  O\left(\sum_{i=1}^{k_j-1} a_i\right),
\end{split}
\end{equation}

where we used Lemma \ref{lem_case_ud1} and the Denjoy--Koksma inequality in the last line. Let $M_j := \lfloor c  a_{k_j }q_{k_j-1} \rfloor $ and, for $\delta > 0$, we define the set $A_j^{ \delta} :=  \Bigg \{ 1 \leq N \leq M_j : \delta \leq \frac{b_{k_j-1}(N)}{a_{k_j}}  \Bigg\}  $. 
We note that $\lim_{\delta \rightarrow 0 } | A_j^{ \delta} | = M_j $. Thus, fixing $\varepsilon> 0$, we can choose $\delta > 0$ such that $ \frac{| A_j^{ \delta} |}{M_j} \geq 1- \varepsilon$ for all sufficiently large $j$. Let $C_1, C_2 > 0$ such that, if $ a_{k_j} \geq \tilde{\psi}(k_j) = C_1 \psi( C_2 k_j)$ it follows that $ b_{k_j -1}c' + O\left( \sum_{i=1}^{k_j -1} a_i  \right) \geq \psi( \log N)$ for all $N \in A_j^{\delta} $. We note that $C_1, C_2$ only depend on $ \delta > 0$, since $k_j$ is chosen such that $ \sum_{i=1}^{k_j -1 } a_i \leq 2 k_j \log (k_j) = o( \psi( k_j) ) $. This yields
\begin{align*}
\frac{ \left | \left \{  1 \leq N \leq M_j : S_N( f, \alpha) \geq \psi( \log N )  \right \}  \right | }{M_j} & \geq  \frac{ \left | \left \{  N \in A_j^{\delta} : b_{k_j-1}c' +  O\left(\sum_{i=1}^{k_j-1} a_i\right) \geq \psi( \log N ) \right \}   \right |}{M_j} \\
& \geq \frac{ \left | \left \{  N \in A_j^{\delta} :  a_{k_j }   \geq \tilde{\psi}( k_j  )  \right \} \right |}{M_j} \\
& = \frac{|A_j^{ \delta} | }{M_j} \geq 1- \varepsilon. 
\end{align*}

By taking the limes inferior as $j \rightarrow \infty $ and letting $\varepsilon\rightarrow 0$, we get
\begin{equation*}
\liminf_{j \rightarrow \infty} \frac{ \left | \left \{  1 \leq N \leq M_j : S_N( f, \alpha) \geq \psi( \log N ) \right \}  \right |}{M_j}  = 1.
\end{equation*}
This shows the claimed upper density $1$ in Theorem \ref{ThmUpperDensDiscFunctionsRefined} for the set $ \{N \in \N \:\ S_N( f, \alpha) \geq \psi( \log N )  \}$ in case of $ \sum_{i=1}^{\nu } A_i \neq  0$ or $ \sum_{i=1}^{ \nu} c_i ( x_{i} - x_{i-1}) \neq 0 $.

In order to prove the first part of Theorem \ref{ThmKoksRefined}, let $0 < r< 1$ be fixed. Choosing $\delta = \delta(r)$ sufficiently small such that $ \frac{| A_j^{ \delta} |}{M_j} \geq r$, we can deduce from \eqref{pf_ud_1}
that, for $N \in  A_j^{ \delta}$, we have
\[S_N(f,\alpha) \geq c'\delta a_{k_j} + O\left( \sum_{i=1}^{k_j -1} a_i  \right).\]
By choosing $C(r):= \frac{c'\delta}{2}$, the first statement of Theorem \ref{ThmKoksRefined} follows, since the sequence $(k_j)_{j \in \N}$ is chosen such that $a_{k_j}$ dominates $\sum_{i =1}^{k_j-1} a_i$.\\

\par{}
Now we prove the remaining parts of Theorem \ref{ThmUpperDensDiscFunctionsRefined}, where we need to show that if $ \sum_{i=1}^{\nu } A_i = \sum_{i=1}^{ \nu} c_i ( x_{i} - x_{i-1}) =0  $, then the set $ \left \{ N \in \N : S_N(f, \alpha) \geq \psi( \log N )  \right \}$ has positive upper density. Let $ c, c', d, \alpha_1, \beta_1, \gamma_1, \delta_1$ be as in Lemma \ref{lem_case_posud}. By Lemma \ref{LemInfiniteSets}, there are infinitely many $K$ such that

\begin{align*}
&  \left \{ K \in \N \Bigg |  K \equiv \alpha_1 \Mod{\beta_1} , q_{K-1}  \equiv  \gamma_1 \Mod{\delta_1}, \psi(K) < a_K < K^2, \sum_{i=1}^{K-1}a_i \leq 2  K \log K  \right \}.
\end{align*}

Denote by $(k_j)_{j \in \mathbb{N}}$ the increasing sequence of integers such that the above holds.
Now let $N \in \N$ with $ c'a_{k_j}q_{k_j} \leq N \leq ca_{k_j}q_{k_j}$ be arbitrary. Thus, we can write $N = b_{k_j-1}q_{k_j-1} + N'$ where $N' < q_{k_j-1}$ and  $c a_{k_j} \leq b_{k_j-1} \leq a_{k_j}$. 
Arguing as in \eqref{pf_ud_1}, we obtain by Lemma \ref{lem_case_posud}

\begin{equation}\label{pf_pos_ud}
S_N(f,\alpha) \geq b_{k_j-1}d +  O\left(\sum_{i=1}^{k_j-1} a_i\right).
\end{equation}

Let $M_j := \lfloor c a_{k_j }q_{k_j-1} \rfloor $ and $A_j :=  \left \{ 1 \leq N \leq N_j : c'  \leq \frac{b_{k_j-1}(N)}{a_{k_j}} \leq c  \right\}  $. We note that there exists $r_0 > 0$ such that $\frac{|A_j|}{M_j} \geq r_0 > 0 $ for all sufficiently large $j \in \N$. Similar to the first part of this proof, we get
\begin{align*}
\frac{ \left | \left \{  1 \leq N \leq M_j : S_N( f, \alpha) \geq \psi( \log N )  \right \}  \right | }{M_j} & \geq  \frac{ \left | \left \{  1 \leq N \leq M_j : b_{k_j-1} d +  O\left(\sum_{i=1}^{k_j-1} a_i\right) \geq \psi( \log N ) \right \}   \right |}{M_j} \\
& \geq  \frac{ \left | \left \{  N \in A_j : b_{k_j-1} d +  O\left(\sum_{i=1}^{k_j-1} a_i\right) \geq \psi( \log N ) \right \}   \right |}{M_j} \\
& =  \frac{|A_j | }{M_j} \geq  r_0. 
\end{align*}
By taking the liminf as $j \rightarrow \infty $, we obtain
\begin{equation*}
\liminf_{j \rightarrow \infty} \frac{ \left | 1 \leq N \leq M_j : S_N( f, \alpha) \geq \psi( \log N )  \right |}{M_j} \geq r_0  > 0.
\end{equation*}
This shows the claimed positive upper density in Theorem \ref{ThmUpperDensDiscFunctionsRefined} for the set $ \{N \in \N \:\ S_N( f, \alpha) \geq \psi( \log N )  \}$ in case of $ \sum_{i=1}^{\nu } A_i = \sum_{i=1}^{ \nu} c_i ( x_{i} - x_{i-1}) = 0 $. 

To prove the second statement of Theorem \ref{ThmKoksRefined}, we see that, for $N \in A_j$, by \eqref{pf_pos_ud}, we have 
$S_N(f,\alpha) \geq dc' a_{k_j} + O\left(\sum_{i=1}^{k_j-1} a_i\right)$. Since $a_{k_j}$ dominates $ \sum_{i=1}^{k_j-1} a_i$ by construction and $\frac{|A_j|}{M_j} \geq r_0 > 0 $, the statement follows immediately.

\section{Functions with logarithmic singularities}
\subsection{Heuristic of the proofs}

We will briefly line out the main ideas of the proof of Theorems \ref{logthm} and Theorem \ref{koks_sharp_log}.
Again, we are using that, for almost every $\alpha \in [0,1)$, $\sum_{i =1}^{K-1} a_i = o(a_K)$ for infinitely many $K \in \N$. Here, $K$ will always satisfy this property. For $q_{K-1} < N < q_K$ and $N = b_{K-1}q_k + N'$, $N' < q_{K-1}$, we get rid of $ S_{N'}(\ldots)$ by an application of the Denjoy--Koksma inequality with singularity \eqref{koksma_bound_sing}. We make sure to stay away from the singularity $x_1 = \frac{r}{s}$ by the fact that
if $\lVert N\alpha - \frac{r}{s} \rVert$ is small, then so is $\lVert sN\alpha \rVert$ (Proposition \ref{PropDistanceNalphaToq}). Thus, 
we can morally work with the homogeneous case of Diophantine approximation and the corresponding metric theory gives sufficient estimates.

Again, we analyze $S_{(b+1)q_{K-1}}(f,\alpha) - S_{bq_{K-1}}(f,\alpha)$ for every $b \leq b_{K-1}-1$ and observe that \[\left\{\{n\alpha + bq_{K-1}\alpha\}\right\}_{n =1}^{q_{K-1}}\approx
\left\{\frac{j + \frac{b_{K-1}}{a_K}}{q_{K-1}}\right\}_{j =1}^{q_{K-1}}.\] 

In the asymmetric case, we see that $f(x) = \log\{x\}$ is monotonically increasing on $[0,1)$.
Comparing $f\left(\frac{j + \frac{b_{K-1}}{a_K}}{q_{K-1}}\right)$ with $\int\limits_{j/q_{K-1}}^{(j+1)/q_K} f(x) \,\mathrm{d}x$, the value of $\frac{b_{K-1}}{a_K}$ is decisive and, for some $c,d > 0$ and $\frac{b}{a_K}  \in [0,c]$, this leads to an estimate (see Lemma \ref{LemMainTermEstRationalShift}) of the form
\[(-1)^K\left(S_{(b+1)q_{K-1}}(f,\alpha) - S_{bq_{K-1}}(f,\alpha)\right) \gtrapprox d \log q_K \gg
\log N.\] Then the proof can be concluded similarly to the proof of Theorem \ref{ThmUpperDensDiscFunctionsRefined}.

In the symmetric case $f(x) = \log \lVert x \rVert$, we see that, for $j \leq q_{K-1}/2$, we have
$f\left(1 - \frac{j -1 + \frac{b_{K-1}}{a_K}}{q_{K-1}}\right)
= f\left(\frac{j + \left(1 - \frac{b_{K-1}}{a_K}\right)}{q_{K-1}}\right)
$. So, the terms $f\left(\frac{j + \frac{b_{K-1}}{a_K}}{q_{K-1}}\right)-\int\limits_{j/q_{K-1}}^{(j+1)/q_K} f(x) \,\mathrm{d}x$ and $f\left(\frac{1- (j + \frac{b_{K-1}}{a_K})}{q_{K-1}}\right)- \int\limits_{1 - (j-1)/q_{K-1}}^{1 - j/q_{K-1}} f(x) \,\mathrm{d}x$ are of opposite sign and lead to some cancellation (Lemma \ref{final_symm_log_estim}). This cancellation is responsible for the different behaviour of symmetric and asymmetric singularities.

\subsection{Asymmetric logarithmic singularities}

\begin{proposition}\label{sum_vs_int_log}
\label{sum_vs_int}
    Let $x_j = \frac{j + \varepsilon_{j}}{q_{\ell}}, 0 \leq j \leq q_{\ell} -1$, for $0 <\varepsilon_{j}< 1$, where $ \ell \in \N $. Then, we have

    \[\sum_{j=0}^{q_{\ell}-1}\log(x_j) - q_{\ell}\int_{0}^{1} \log(x) \,\mathrm{d}x = \left(\sum_{j=1}^{q_{\ell} -1}\frac{\varepsilon_{j}-1/2}{j}\right) + \log(\varepsilon_{0}) + O(1)\]
    with the implied constant being absolute, independent of $\varepsilon_{j}$.
\end{proposition}

\begin{proof}
    For $j \geq 1$, we have
    \[q_{\ell}\int_{j/q_{\ell}}^{(j+1)/q_{\ell}} \log(x) \,\mathrm{d}x
    = (j+1)\log\left(\frac{j+1}{q_{\ell}}\right) - (j+1) - j\log\left(\frac{j}{\ql}\right) + j
    = \log\left(\frac{j+1}{q_{\ell}}\right) + j\log\left(1 + \frac{1}{j}\right) - 1.
    \]

So, we obtain
\[
\log(x_j) - q_{\ell}\int_{j/q_{\ell}}^{(j+1)/q_{\ell}} \log(x) \,\mathrm{d}x =\log \left(\frac{j + \varepsilon_{j}}{j +1}\right) - j\log\left(1 + \frac{1}{j}\right) + 1.\]

By the Taylor expansion $\log(1 + x) = x - x^2/2 + O(x^3)$, we get

\[\log \left(\frac{j + \varepsilon_{j}}{j +1}\right) - j\log\left(1 + \frac{1}{j}\right) + 1
= \frac{\varepsilon_{j}-1/2}{j} + O\left(\frac{1}{j^2}\right).
\]
 For $j = 0$, we get
$q_{\ell}\int_{0}^{1/q_{\ell}} \log(x) \,\mathrm{d}x = - \log(\ql) - 1$
and thus, 
\[\log(x_0) - q_{\ell}\int_{0}^{1/q_{\ell}} \log(x) \,\mathrm{d}x = \log(\varepsilon_{0}) +1.\]

Combining the obtained estimates yields
 \[\sum_{j=0}^{q_{\ell}-1}\log(x_j) - q_{\ell}\int_{0}^{1} \log(x) \,\mathrm{d}x = \left(\sum_{j=1}^{q_{\ell} -1}\frac{\varepsilon_{j}-1/2}{j}\right) + \log(\varepsilon_{0}) + O(1).\]
\end{proof}

\begin{proposition}
\label{PropDistanceNalphaToq}
    Let $\frac{r}{s} \in [0,1) \cap \Q $ and let $1 < N < \frac{q_K}{s}$. Then,
    \[\left\lVert N\alpha - \frac{r}{s}\right\rVert > \frac{\lVert q_{K+1}\alpha\rVert}{s}.\]
\end{proposition}

\begin{proof}
Assume to the contrary that $\left\lVert N\alpha - \frac{r}{s}\right\rVert \leq \frac{\lVert q_{K+1}\alpha\rVert}{s}$.
Then, we have

\[\lVert sN\alpha\rVert = \left\lVert sN\alpha - s \frac{r}{s}\right\rVert = 
\left\lVert s\left(N\alpha - \frac{r}{s}\right)\right\rVert \leq s \left\lVert N\alpha - \frac{r}{s}\right\rVert
\leq \lVert q_{K+1}\alpha \rVert.\]

Since $sN < q_{K+1}$, this is a contradiction to the best approximation property of $q_{K+1}$: There would exist an integer $N' < q_{K+1}$ such that $q_{K-1} \lVert N'\alpha \rVert \leq \lVert q_{K+1}\alpha \rVert$.
\end{proof}

\begin{proposition}(Error term estimate for a rational shift)
\par{}
    \label{PropErrorEstimateLogSingularityRationalShift}
    Let $f(x)= \log( \{ x - q \} )$, where $q = \frac{r}{s} \in [ 0,1)$ is a rational number. Let $N \in \N$ with $ N  < \frac{q_K}{s}$ and let $N= b_{K-1} q_{K-1} + N'$, where $0 \leq N' < q_{K-1}$. Then, we have
    \begin{equation*}
    \left | S_{N'}( f, \alpha , b_{K-1}q_{K-1} \alpha ) \right | \ll \log( q_{K+1} ) \sum_{\ell=1}^{K-1} a_{\ell}.
    \end{equation*}
\end{proposition}

\begin{proof}

Define $A_{N'} := \left ( q -  \frac{ \| q_K \alpha \| }{s}, q +  \frac{\| q_K \alpha \|}{s}\right )$. Using Proposition \ref{PropDistanceNalphaToq} we see $ n \alpha + b_{K-1}q_{K-1} \alpha \notin A_{N'}$ for all $n \leq N'$. Thus, the Denjoy--Koksma inequality with singularity (Proposition \ref{denj_koks_log}) yields
\begin{equation*}
\left | S_{N'}( f, \alpha , b_{K-1}q_{K-1} \alpha ) \right | \ll \sup_{x \in [0,1)\setminus A_{N'}} \lvert f(x)\rvert  \sum_{i =1}^{K-1}a_i+ q_{K-1} \left\lvert\int_{A_{N'}} f(x) \,\mathrm{d}x\right\rvert.
\end{equation*}
We have $\sup_{x \in [0,1)\setminus A_{N'}} \lvert f(x)\rvert = \log \frac{\| q_K \alpha \| }{s} \ll \log q_{K+1}$. Further, we use the estimate $q_{K-1} \left \lvert\int_{A_{N'}} f(x) \,\mathrm{d}x\right\rvert \ll \log (q_{K+1})$ to obtain the desired result.
\end{proof}

\begin{lem}(Main term estimate for a rational shift)
\label{LemMainTermEstRationalShift}
\par{}
Let $\delta > 0$, $q = \frac{r}{s} \in [ 0,1) \cap \Q $ and define $f(x) := \log( \{ x - q\} )$ for $x \in [0,1)$. Further, let $K \in \N$ with $ q_{K-1} \equiv 0 \Mod{s}$ and choose $N \in \N$ with $q_{K-1} < N < q_{K}$ and $\delta a_K < b_{K-1}(N) < \frac{a_K}{4}$. Then, if $a_K $ is sufficiently large, we get
\begin{equation*}
(-1)^K S_{b_{K-1}q_{K-1}}(f, \alpha ) \gg \delta  a_K \log(q_{K-1}).
\end{equation*} 
\end{lem}

\begin{proof}
By definition, we have

\begin{align*}
S_{b_{K-1}q_{K-1}}(f, \alpha) & = \sum_{n=1}^{b_{K-1} q_{K-1}} \log \left( \left \{ n \alpha - q \right \} \right) + b_{K-1} q_{K-1} \\
&= \sum_{b = 0}^{b_{K-1}-1} \left(  \sum_{n=1}^{q_{K-1}} \log \left( \left \{ n \alpha + b q_{K-1} \alpha - q \right \} \right) + q_{K-1} \right)\\
& = \sum_{b = 0}^{b_{K-1}-1} \left(  \sum_{n=1}^{q_{K-1}} \log \left( \left \{ n \alpha + (-1)^{K-1}b \delta_{K-1} \alpha - q \right \} \right) + q_{K-1} \right),
\end{align*}
where we recall that $ \delta_{K-1} = | \alpha q_{K-1} - p_{K-1}| = (-1)^{K-1} ( \alpha q_{K-1} - p_{K-1})$. For $1 \leq n \leq q_{K-1}$ and $0 \leq b \leq b_{K-1}-1$, we get

\begin{align*}
\left \{n \alpha - q +(-1)^{K-1} \left (b\delta_{K-1} \right ) \right \} & = 
\left \{  n \frac{p_{K-1}}{q_{K-1}} - q  + \frac{(-1)^{K-1}}{q_{K-1}} \delta_{K-1} \left(b q_{K-1} + n  \right)\right\} \\
& = 
\left\{\frac{np_{K-1} - q q_{K-1}}{q_{K-1}} + \frac{(-1)^{K-1}}{q_{K-1}} \left(\underbrace{\delta_{K-1} (bq_{K-1} +n )}_{=:\varepsilon_{b,n}} \right)\right\} \\
& = \left\{\frac{np_{K-1}-n'}{q_{K-1}} + \frac{(-1)^{K-1}}{q_{K-1}} \varepsilon_{b,n}\right\} .
\end{align*}
We introduced $n' :=  q q_{K-1} $ which is an integer since $s | q_{K-1}$ with $0 \leq n' \leq q_{K-1} -1 $. Now observe that $\delta_{K-1} > 0$ by definition and since $b_{q_{K-1}} < \frac{a_K}{4}$, we have $bq_{K-1} +n \leq b_{q_{K-1}}q_{K-1} \leq (1/2 - \delta)a_Kq_{K-1}$. Thus, for any $ 0 \leq n \leq q_{K-1} -1$, we get 
\[0 < \varepsilon_{b,n} < \frac{a_k}{4}q_{K-1}\delta_{K-1} \leq 1/4,\]
where we used that $q_{K-1} \delta_{K-1} \leq 1/a_K$. First, we assume that $K$ is odd implying $(-1)^{K-1} = 1$. 
We apply Proposition \ref{sum_vs_int} with $x_n = \left\{\frac{n}{q_{K-1}} + \frac{\varepsilon_{n}}{q_{K-1}}  \right\}$, where $\varepsilon_{n} = \varepsilon_{b,((np_{K-1}^{-1}+n') \Mod{q_{K-1}} )}$ for $1 \leq n \leq q_{K-1}-1$ and $\varepsilon_{0} = \varepsilon_{b,(1-q)q_{K-1}}$. This leads to

\begin{align*}
\sum_{n=1}^{q_{K-1}} \log \left( \left \{ n \alpha + (-1)^{K-1}b \delta_{K-1} \alpha - q \right \} \right) + q_{K-1} & = 
\sum_{n=1}^{q_{K-1}} \log \left( \left \{ n \alpha + (-1)^{K-1}b \delta_{K-1} \alpha - q \right \} \right) - q_{K-1} \int_{0}^1 \log(x)dx \\
& = \left(\sum_{n=1}^{q_{K-1} -1}\frac{\varepsilon_{n}-1/2}{n}\right) + \log(\varepsilon_{0}) + O(1) \\
& \leq -\frac{1}{4}\log(q_{K-1}) + \log(\varepsilon_{0}) + O(1) \\
& \leq -\frac{1}{8}\log(q_{K-1}).
\end{align*}

We used that $ \varepsilon_0 \leq \frac{1}{4} < 1$ and hence $\log( \varepsilon_{0}) \leq 0$. Moreover, we applied the rough estimate $ O(1) \leq \frac{1}{8} \log (q_{K-1} )$. By summing over all $b =  0, \ldots, b_{K-1} -1 $, we obtain

\[ S_{b_{K-1}q_{K-1}}(f, \alpha)  \leq - \frac{1}{8}  b_{K-1} \log(q_{K-1}) \leq - \frac{\delta}{8} a_K \log(q_{K-1}) \ll - \delta  a_K \log( q_{K-1} ), \]

where we also used the assumption $ b_{K-1} \geq \delta a_K$. By rewriting, we finally get
\begin{equation*}
 - S_{b_{K-1}q_{K-1}}(f, \alpha)  \gg \delta a_K \log( q_{K-1} ),
\end{equation*}
as claimed.
\par{}
Now let $K$ be even. We apply Proposition \ref{sum_vs_int} with $x_n = \left\{\frac{n}{q_{K-1}} + \frac{\varepsilon_{n}}{q_{K-1}}  \right\}$, where $\varepsilon_{n} = 1 - \varepsilon_{b,(n p_{K-1}^{-1}+n' + 1) \Mod{q_{K-1}}}$ for $1 \leq n \leq q_{K-1}-1$ and $ \varepsilon_{0} = 1 - \varepsilon_{b,q q_{K-1} }$. Similar to before, we obtain

\begin{align*}
 S_{b_{K-1}q_{K-1}}(f, \alpha) & = S_{b_{K-1}q_{K-1}}(f, \alpha) \\
& \geq \frac{1}{8} b_{K-1} \log(q_{K-1}) \\
& \gg  \delta  a_K \log( q_{K-1} ).
\end{align*}
This finishes the proof.
\end{proof}

\subsection{Symmetric logarithmic singularities}

\begin{lem}\label{final_symm_log_estim}
Let $f(x) = \log \left \lVert x - \frac{r}{s} \right \rVert$, where $\frac{r}{s} \in [0,1) \cap \Q$. Then, for almost every $\alpha \in [0,1)$,
we have
\[\lvert S_N(f,\alpha) \rvert \ll (\log N )^2 \log \log N . \]
\end{lem}

\begin{proof}
Writing $N = \sum_{\ell=0}^{K-1} b_{\ell}q_{\ell}$ in its Ostrowski expansion with $b_{K-1} \neq 0$, 
we obtain the decomposition 

\[S_N(f,\alpha) = S_{N'}( f , \alpha) +  S_{b_{K_0-1} q_{K_0-1}}(f, \alpha, N' \alpha ) + S_{N''}(f, \alpha, (N' + b_{K_0 -1}q_{K_0 -1}) \alpha),
\]

where $K_0 = \argmax_{\ell = 1, \ldots, K}  a_{ \ell} $, $N' = \sum_{ \ell =K_0 }^{K-1} b_{\ell} q_{\ell}$ and $N'' = \sum_{\ell=0}^{K_0 -2} b_{ \ell} q_{\ell}$. By the Denjoy--Koksma inequality with singularity (see \eqref{koksma_bound_sing} in Proposition \ref{denj_koks_log}), we can bound $S_{N'}( f, \alpha) $ by
\begin{align*}
\left | S_{N'}( f, \alpha) \right | \ll \sup_{x \in [0,1) \setminus A_{N'}} | f(x)| \sum_{i=K_0}^{K-1} b_i + N' \left | \int_{A_{N'}} f(x) dx \right | ,
\end{align*}
where we choose $ A_{N'} = \left ( q - \min_{n < q_{K}} \| n \alpha - q \|, q + \min_{n < q_{K}} \| n \alpha - q \| \right ) $. This ensures $ \{ n \alpha -q \} \notin A_{N'} $ and we have
\begin{align}
\notag 
 \sup_{x \in [0,1)\setminus A_{N'}} \left | \log \left ( \| x -q \| \right ) \right | & \leq \left | \log \left ( \min_{n < q_{K}} \| n \alpha - q \| \right ) \right | \\
 \notag 
 & \leq \left | \log \left ( \min_{n < q_{K +1} / s}  \| n \alpha - q \| \right ) \right |  \\
 \label{EqMinqks}
 & \leq \left | \log \left ( \frac{\| q_{K+s+1} \alpha \|}{s} \right ) \right |  \\
 \notag 
 & \ll  \log \left ( q_{K+s+2}  \right ) \\
 \notag 
 & \ll K,
\end{align}
where we have used that $ \log(q_{K}) \ll K$ by \eqref{Eq_size_of_q_k} and we used Proposition \ref{PropDistanceNalphaToq}. A simple calculation reveals
\begin{align*}
    \left | \int_{A_{N'}} f(x) dx \right| & = 2 \left | \min_{n < q_K} \| n \alpha - q\| \left( \log \left( \min_{n <  q_K} \| n \alpha - q\|  \right)   \right) \right | \\
    & \ll \frac{K}{q_K},
\end{align*}
where we used that $ \min_{n < q_K} \| n \alpha - q\| \leq \frac{1}{q_K}$ and $\left | \log \left( \min_{n \leq q_K} \| n \alpha - q \|  \right) \right | \ll K  $, as shown before. In total, we get
\begin{align*}
\left | S_{N'}( f, \alpha) \right | & \ll K \sum_{i=K_0 }^{K-1} b_i + N' \frac{K}{q_K} \\
& \ll K \sum_{i=K_0 +1}^{K} a_i \\
& \ll K^2 \log K,
\end{align*}
where the estimate in the last line uses \eqref{Eq_trimmed_sum}. Analogously, one obtains the same bound for $S_{N''}(f, \alpha, (N' + b_{K_0 -1}q_{K_0 -1}) \alpha)$, i.e., we get
\begin{equation*}
    \left | S_{N''}(f, \alpha, (N' + b_{K_0 -1}q_{K_0 -1}) \alpha) \right |  \ll K^2 \log K.
\end{equation*}
We are now left with $ S_{b_{K_0-1} q_{K_0-1}}(f, \alpha, N' \alpha ) $, where we will show that 
\begin{equation*}%\label{estimate_symm_log_l}
\left\lvert S_{b_{K_0-1}q_{K_0-1}}(f,\alpha,N' \alpha ) \right\rvert \ll K^2 \log K.
\end{equation*}
By definition, we have
\begin{align*}
S_{b_{K_0-1}q_{K_0-1}}(f,\alpha,N' \alpha ) & = \sum_{n=1}^{b_{K_0 -1}q_{K_0 -1}} \log \left(  \left \| n \alpha - \frac{r}{s} + N' \alpha \right \|  \right) + b_{K_0-1}q_{K_0-1} \\
& = \sum_{b=0}^{b_{K_0 -1} -1} \left( \sum_{n=1}^{q_{K_0 -1}}   \log \left( \left \| n \alpha - \frac{r}{s} + N' \alpha + b q_{K_0 -1} \alpha   \right \|  \right) + q_{K_0-1} \right) \\
& = \sum_{b=0}^{b_{K_0 -1} -1} \left(  \sum_{n=1}^{q_{K_0 -1}}  \log \left( \left \|  n \alpha - \frac{r}{s} + N' \alpha + b (-1)^{K_0-1} \delta_{K_0 -1}   \right \| \right) + q_{K_0-1} \right),
\end{align*}
where we recall that $\delta_{K_0-1} = (-1)^{K_0-1} \left( q_{K_0-1} \alpha - p_{K_0-1} \right)$. 
\par{}
We first assume that $K_0 $ is odd. We can write $ N' \alpha  - \frac{r}{s} = \frac{m}{q_{K_0-1}} + \frac{r'}{q_{K_0-1}}$ where $m \in \mathbb{Z}$ and $0 \leq r' < 1$.
We observe that, for any $0 \leq b \leq b_{K_0-1}$, we have $0 \leq bq_{K_0-1}\delta_{K_0-1} + r' < 2$. For the following analysis, we define the quantity $d_{b} := bq_{K_0-1}\delta_{K_0-1} + r'$ and the sets $B_1, B_2, B_3$ as

\[\begin{split}B_1 &:= \{0 \leq b \leq b_{K_0-1}-1:  d_{b+1} < 1\},\\
B_2 &:= \{0 \leq b \leq b_{K_0-1}-1:  d_b > 1\},\\
B_3 &:= \{0,\ldots,b_{K_0-1}-1 : d_b < 1 < d_{b+1} \}.
\end{split}
\]
Since $d_b$ is irrational for all $0 \leq b \leq b_{K_0-1} -1 $, the sets $B_1, B_2, B_3$ form a partition of $ \{ 0, \ldots, b_{K_0-1} -1 \} $. 
We first assume $b \in B_1$, i.e. $d_{b+1} < 1$. We see that, for all $n=1, \ldots , q_{K_0-1}$, we have
\[\left\{(n + bq_{K_0-1})\alpha + N' \alpha  - \frac{r}{s}\right\}
= \left\{\frac{np_{K_0-1} + m + d_{b} + n\delta_{K_0-1}}{q_{K_0-1}}\right\}.
\]
Since $p_{K_0-1} $ and $ q_{K_0 -1}$ are coprime, the map $j(n) = np_{K_0-1} + m \pmod{q_{K_0-1}}$ is bijective with inverse $n(j)$. Thus, we can introduce the quantities
 \[y_j := \frac{j + d_{b} + n(j)\delta_{K_0-1}}{q_{K_0-1}}, \quad
 j = 0,\ldots, q_{K_0-1}-1.
 \]
Since $  d_{b} + n(j)\delta_{K_0-1} \leq d_{b+1} < 1$ by assumption, it holds that $0 \leq  y_j < 1$ for all $j=0, \ldots, q_{K_0-1} -1$. Further, we have the following equality of sets
\begin{equation*}
 \{ y_j : j=0, \ldots, q_{K_0-1} -1 \} = \left\{ \left\{ \frac{np_{K_0-1} + m + d_{b} + n\delta_{K_0-1}}{q_{K_0-1}} \right\} : n=1, \ldots, q_{K_0-1}  \right\}.
\end{equation*}
The previous arguments reveal that, for $b \in B_1$, we can write 
\begin{align*}
     \sum_{n=1}^{q_{K_0 -1}}   \log \left( \left \|  n \alpha - \frac{r}{s} + N' \alpha + b (-1)^{K_0-1} \delta_{K_0 -1}   \right \| \right) + q_{K_0-1} & = \sum_{j=0}^{q_{K_0-1}-1} \log \| y_j \| + q_{K_0 -1} \\
     & = \sum_{j=0}^{q_{K_0-1}-1} \left ( \log \| y_j \|  -  I_j  \right),
\end{align*}
where we set $I_j:= q_{K_0 -1} \int\limits_{j/q_{K_0-1}}^{(j+1)/q_{K_0-1}} \log(x) \,\mathrm{d}x $ for $j=0, \ldots , q_{K_0 -1}-1$.
In the following, we will compare the value of $\log \lVert y_j \rVert$ to the value of $ I_j$ for all $j=0 , \ldots, q_{K_0 -1}-1$. We start with the case where $1 \leq j \leq \lfloor q_{K_0-1}/2\rfloor - 1$. Then, we have $\lVert y_j \rVert = y_j$ and $ \lVert y_{q_{K_0-1} - j -1} \rVert = 1 - y_{q_{K_0-1} - j -1}$. This leads to
\begin{align*}
\log \lVert y_j\rVert - I_j
& =
\log(y_j) -I_j \\
& =
\log \left(\frac{j + d_{b} + n(j)\delta_{K_0-1}}{j +1}\right) - j\log\left(1 + \frac{1}{j}\right) + 1.
\end{align*}

By the Taylor expansion $\log(1 + x) = x - x^2/2 + O(x^3)$, we have

\[\log \left(\frac{j + d_{b} + n(j)\delta_{K_0-1}}{j +1}\right) - j\log\left(1 + \frac{1}{j}\right) + 1
= \frac{d_{b} + n(j)\delta_{K_0-1}-1/2}{j} + O\left(\frac{1}{j^2}\right).
\]

By the same arguments, we obtain

\[\log \lVert y_{q_{K_0-1}-j-1}\rVert - q_{K_0-1} I_j
= \frac{1 - d_{b} + n(q_{K_0-1}-j-1)\delta_{K_0-1}-1/2}{j} + O\left(\frac{1}{j^2}\right).
\]
So, by combining the two previous estimates, we get

\[
\begin{split}
\left | \log \lVert y_j\rVert + \log \lVert y_{q_{K_0-1}-j-1}\rVert - 2 I_j \right | 
& \leq 
\left\lvert\frac{\left(n(j)-n(q_{K_0-1}-j-1)\right) \delta_{K_0-1}}{j}\right\rvert +  O \left( \frac{1}{j^2} \right)\\
& \leq \frac{1}{j a_{K_0}}+ O\left(\frac{1}{j^2}\right).
\end{split}\]
In the last line, we used the estimate $q_{K_0-1}\delta_{K_0-1} \leq \frac{1}{a_{K_0}}$. It is easy to see that 
\begin{equation*}
\sum_{j= \lfloor q_{K_0-1}/2\rfloor }^{q_{K_0-1} - \lfloor q_{K_0-1}/2\rfloor } \left ( \log( \| y_j \|) - 2 I_j \right)   = O(1),
\end{equation*}
since the number of summands on the left-hand side is bounded by a constant and the $y_j$ are bounded away from $0$ and $1$. 
Thus, we have shown that 
\begin{align*}
\left | \sum_{j=1}^{q_{K_0-1}-2} \log \| y_j \| - I_j  \right | 
&\ll 
\frac{1}{a_{K_0} } \sum_{j=1}^{\lfloor q_{K_0-1}/2\rfloor - 1 } \frac{1}{j} + \sum_{j=1}^{\lfloor q_{K_0-1}/2\rfloor - 1 } \frac{1}{j^2} + O(1) \\
& \ll 
\frac{\log q_{K_0 -1}}{a_{K_0}} + O(1).
\end{align*}

We are now left with the cases $j=0$ and $j=q_{K_0-1} -1$, where we get

\begin{equation*}
\begin{split}
\left\lvert \log \lVert y_0 \rVert + \log \lVert y_{q_{K_0-1} - 1} \rVert 
- 2 I_0\right\rvert &\leq 
\left\lvert \log \lVert y_0 \rVert -  I_0 \right\rvert  + 
\left\lvert  \log \lVert y_{q_{K_0-1} - 1} \rVert 
- I_0 \right\rvert \\
& \leq 
\lvert \log(d_{b} + n(0)\delta_{K_0-1}) \rvert+
\lvert \log(1 - d_{b} - n(q_{K_0-1}-1)\delta_{K_0-1}) \rvert+ 2.
\end{split}
\end{equation*}

We discuss here the first term $ \lvert \log(d_{b} + n(0)\delta_{K_0-1}) \rvert$ in detail, the second term can be treated analogously. Observe that $d_{b} + n(0)\delta_{K_0-1} = q_{K_0-1}\lVert (N' + bq_{K_0 -1} +  n(0))\alpha - \frac{r}{s})\rVert$ by construction of $y_0$.

We claim that there exists at most one $b' \in B_1$ such that
\[ \left \lVert \left (N' + b'q_{K_0 -1} +  n(0) \right ) \alpha - \frac{r}{s}  \right \rVert 
\leq \frac{1}{4q_{K_0}}.\]

Assume to the contrary that there are two integers $b', b'' \in B_1$ with $b' \neq b''$ such that both satisfy this estimate. Then, we get
\begin{align*}
\left \|  \left ( b' - b'' \right) q_{K_0 -1} \alpha  \right \|  \leq \frac{1}{2 q_{K_0}}
\end{align*}
by the triangle inequality, which is an immediate contradiction to the best approximation property of $q_{K_0 -1} $. Thus, for the only possible $b' \in B_1$, we get
\begin{align*}
q_{K_0 -1} \left \lVert \left (N' + b'q_{K_0 -1} +  n(0) \right ) \alpha - \frac{r}{s}  \right \rVert & \geq q_{K_0 -1} \min_{ n < q_{K}} \left \lVert n \alpha - \frac{r}{s}  \right \rVert \\
& \geq \frac{q_{K_0 -1}}{q_{K +s +1}} \\
& \geq \frac{1}{q_{K +s +1}},
\end{align*}
where the estimate in the second last step can be argued analogously to \eqref{EqMinqks}. For all $b \in B_1$ with $b \neq b'$, we have 
\begin{align*}
    q_{K_0 -1} \left \lVert \left (N' + bq_{K_0 -1} +  n(0) \right ) \alpha - \frac{r}{s}  \right \rVert & \geq \frac{q_{K_0 -1}}{4 q_{K_0}} \\
    & \geq \frac{1}{4 ( a_{K_0 } +1)  }.
\end{align*}

Thus, by combining the estimates we obtained, we get
\begin{align*}
\left | \sum_{b \in B_1} \sum_{n=1}^{q_{K_0 -1}}  f(  n \alpha + b q_{K_0 -1} \alpha  + N' \alpha ) \right | & = \left |  \sum_{b \in B_1}  \sum_{j=0}^{q_{K_0 -1}} \left( \log \| y_j \| - I_j \right) \right |  \\
& \leq \left | \sum_{b \in B_1} \sum_{j=1}^{q_{K_0 -2}} \left(  \log \|y_j \| - I_j \right) \right | + \left | \sum_{b \in B_1}  \sum_{j \in \{ 0, q_{K_0 -1}  \} }  \left( \log \| y_j \| - I_j \right) \right |  \\
& \ll \log q_{K_0 -1} a_{K_0} + \left | \sum_{b \in B_1, b \neq b'}  \sum_{j \in \{ 0, q_{K_0 -1}  \} } \left( \log \| y_j \| - I_j \right) \right |   + \left | \sum_{j \in \{ 0, q_{K_0 -1}  \} } \left(  \log \| y_j( b') \|  - I_j \right) \right | \\
& \ll \log q_{K_0 -1} a_{K_0} +  \sum_{b \in B_1, b\neq b'} \log a_{K_0} + \log q_{K + s + 1} \\
& \ll a_{K_0} (\log a_{K_0} + \log q_{K_0 -1}) \\
& \ll K^2 \log (K).
\end{align*}
Here we used that $ | B_1| \leq b_{K_0 -1 } \leq a_{K_0}$, $ \log q_{K +s +1} \ll K$ by \eqref{Eq_size_of_q_k} and by \eqref{Eqbernstein}, $a_{K_0} \leq K^2$ if $K$ is sufficiently large. For the set $B_2$, a similar analysis leads to the same asymptotic bound, i.e., we get
\begin{equation*}
    \left | \sum_{b \in B_2} \sum_{n=1}^{q_{K_0 -1}}  f(  n \alpha + b q_{K_0 -1} \alpha  + N' \alpha ) \right |  \ll K^2 \log K.
\end{equation*}
The set $B_3$ contains at most $1$ element $\bar{b} \in \{0 , \ldots , b_{K_0 -1} -1 \} $ and thus, we can write
\begin{align*}
\left | \sum_{b \in B_3} \sum_{n=1}^{q_{K_0 -1}}  f(  n \alpha + b q_{K_0 -1} \alpha  + N' \alpha ) \right | 
& \leq 
\left | \sum_{n=1}^{q_{K_0 -1}}  \log \left ( \left  \|    n \alpha + x \right \|  \right ) \right |, 
\end{align*}
where $x := \bar{b} q_{K_0 -1} \alpha  + N' \alpha$. Applying the Denjoy--Koksma inequality with singularity in the form of \eqref{koksma_bound_sing}, we obtain
\begin{equation*}
 \left | \sum_{n=1}^{q_{K_0 -1}}  \log \left ( \left \|    n \alpha + x  \right \| \right ) \right | \ll K^2 \log K ,
\end{equation*}
as we did for $S_{N'}(f ,\alpha)$ at the beginning of this proof. Combining the estimates for $B_1, B_2, B_3$, we can deduce that 
\begin{equation*}
\left | S_{b_{K_0 -1}q_{K_0 -1}} (f, \alpha, N' \alpha) \right |  \ll K^2 \log K,
\end{equation*}
which finishes the proof for odd $K_0$. The case where $K_0$ is even can be handled under minor modifications. In total, we have shown that, for $ q_{K-1} \leq N < q_K$ 
\begin{align*}
 \vert S_N(f, \alpha) \rvert   & \ll  K^2 \log K \\
 & \ll (\log N)^2 \log \log N.
\end{align*}
\end{proof}

\subsection{Proof of Theorem \ref{logthm} and Theorem \ref{koks_sharp_log}}
We start by proving (ii) of Theorem \ref{logthm}, where the Birkhoff sum $S_N(f, \alpha, q)$ is generated by a function $f : \R \rightarrow \R$ with symmetric logarithmic singularity at a rational, i.e. $f$ is of the form $ f(x) = c \log \| x - x_1 \|  + t(x)$, where $c \neq  0$ ,$x_1 = \frac{r}{s} \in [0,1) \cap \Q $ and $t$ is of bounded variation. 
Without loss of generality, we can assume that $q=0$ because otherwise we just set $ \tilde{x}_1 := x_1 - q$.  
Let $N \in \N$ with Ostrowski expansion $N = \sum_{i =1}^{K-1} b_i q_i$.
By the Denjoy--Koksma inequality, we obtain
$S_N(t,\alpha) \leq \Var(t) \sum_{i=1}^{K}a_i \ll K^2$.
Thus by Lemma \ref{final_symm_log_estim}, we get
\begin{equation*}
    \left | S_N( f, \alpha)  \right | \ll ( \log N)^2 \log \log  N
\end{equation*}
implying statement (ii) of Theorem \ref{logthm}.
\par{}
Next, we consider the asymmetric case, i.e., where the Birkhoff sum is generated by a function $f$ of the form $f(x)= \underbrace{c_1 \log( \{ x-x_1\} )}_{=: f_1(x)}+ \underbrace{c_2 \log \| x - x_1 \| + t(x)}_{=: f_2(x)} $, where $c_1, c_2 \in \R$ with $c_1 \neq 0$ and $x_1 = \frac{r}{s} \in [0, 1) \cap \Q $ and $t$ is of bounded variation. Again, without loss of generality, it suffices to consider the case where $q=0$.
\par{}
 We start with the case where $\sum_{k=1}^{\infty} \frac{1}{\psi(k)} = \infty$ where we show that the set $ \left \{ N \in \N : S_N(f, \alpha) \geq \log N  \psi( \log(N)) \right\}$ has upper density $1$. Without loss of generality, we can assume that $ \lim_{K \rightarrow \infty } \frac{\psi(K)}{K \log K} = \infty $ since the result then follows also for slower growing $\psi$. Note that $S_N(f, \alpha) = S_N(f_1, \alpha) + S_N(f_2, \alpha)$. By the first part of this proof, we have $|S_N(f_2, \alpha )| \ll (\log N)^2 \log \log N $ and since $ \log N \psi ( \log N )$ dominates $ (\log N) ^2 \log \log N$, it suffices to show that $ \{ N \in \N : S_N( f_1 , \alpha ) \geq \log N \psi ( \log N) \} $ has upper density $1$.
\par{}
First, assume that $c_1 > 0$. By Lemma \ref{LemInfiniteSets}, for almost every $\alpha \in [0,1)$, the set 
\begin{align*}
&  \left \{ K \in \N \Bigg |  K \equiv 0 \Mod{2} , q_{K-1}  \equiv  0 \Mod{s}, \tilde{\psi}(K) < a_K < K^2, \sum_{i=1}^{K-1}a_i \leq 2  K \log K  \right \}
\end{align*}
has infinite cardinality, where $\tilde{\psi}(k) = C_1 \psi( C_2 k)$ with $C_1, C_2 > 0$ specified later. Denote by $(k_j)_{j \in \mathbb{N}}$ the increasing sequence of integers such that the above holds. Define $M_j := \left \lfloor \frac{a_{k_j}}{4s} q_{k_j-1} \right \rfloor $ and note that, for any $1 \leq N \leq M_j$, we can write $N= b_{k_j-1}(N) q_{k_j-1} + N'$ where $N' < q_{k_j-1}$. Moreover, for $\delta > 0$, let $A_j^{\delta}:= \left \{ 1 \leq N \leq M_j : \delta < \frac{b_{k_j-1}(N)}{a_{k_j}} \leq 1 \right\}$. We note that for any $N \in A_j^{\delta}$, the assumptions of Lemma \ref{LemMainTermEstRationalShift} are satisfied, and hence
\begin{equation*}
    S_{b_{k_j-1} q_{k_j -1}}(f_1, \alpha)  \geq c(\delta) a_{k_j} \log( q_{k_j-1}),
\end{equation*}
where $c(\delta)$ is a positive constant only depending on $\delta$. Moreover, by Proposition \ref{PropErrorEstimateLogSingularityRationalShift}, we have
\begin{equation*}
    |S_{N'}(f_1, \alpha, b_{k_j-1}q_{k_j -1} \alpha)| \leq D  \log( q_{k_j +1}) \sum_{\ell=1}^{k_j -1}a_{\ell}, 
\end{equation*}
where $D$ is a positive absolute constant. For $j$ sufficiently large, this leads to

\begin{equation}\label{f1_estim}
\begin{split}
S_N(f_1,\alpha) &= S_{b_{k_j-1} q_{k_j -1}}(f_1,\alpha) + S_{N'}(f_1,\alpha, b_{k_j-1}q_{k_j -1} \alpha) \\
& \geq  c( \delta) a_{k_j} \log( q_{k_j-1}) - D \sum_{\ell=1}^{k_j-1} a_{\ell} \\
& \geq  \frac{c( \delta)}{2} a_{k_j} \log( q_{k_{j} -1}),
\end{split}
\end{equation}
where the inequality in the last line holds since $a_{k_j}$ dominates $ \sum_{i=1}^{k_j -1} a_i $ by construction. Moreover, we have used that $c_1$ in the definition of $f_1$ is positive by assumption and we employed $ \log q_{k_j -1} \gg \log  q_{k_j+1}  $ which holds by \eqref{Eq_size_of_q_k}. 
We note that $\lim_{\delta \rightarrow 0 } | A_j^{ \delta} | = M_j $ and thus, fixing $\varepsilon> 0$, we can choose $\delta > 0$ such that $ \frac{| A_j^{ \delta} |}{M_j} \geq 1- \varepsilon$ for all sufficiently large $j$. Let $C_1, C_2 > 0$ such that, if $ a_{k_j} \geq \tilde{\psi}(k_j) = C_1 \psi( C_2 k_j)$, it follows that $ \frac{c( \delta)}{2} a_{k_j} \log( q_{k_{j} -1}) \geq \log N \psi( \log N)$ for all $N \in A_j^{\delta} $. We note that $C_1, C_2$ only depend on $ \delta > 0$, since $k_j$ is chosen such that $ \sum_{i=1}^{k_j -1 } a_i \leq 2 k_j \log (k_j) = o( \psi( k_j) ) $. This yields
\begin{equation}
\begin{split}
\label{EqUpDens1log}
\frac{ \left | \left \{  1 \leq N \leq M_j : S_N( f_1, \alpha) \geq \log N  \psi( \log N )  \right \}  \right | }{M_j} & \geq  \frac{ \left | \left \{  N \in A_j^{\delta} : \frac{c( \delta)}{2} a_{k_j} \log( q_{k_{j} -1}) \geq \log N  \psi( \log N ) \right \}   \right |}{M_j} \\
& \geq \frac{ \left | \left \{  N \in A_j^{\delta} :  a_{k_j }   \geq \tilde{\psi}( k_j  )  \right \}  \right |}{M_j} \\
& = \frac{|A_j^{ \delta} | }{M_j} \geq 1- \varepsilon.
\end{split}
\end{equation}

By taking the liminf as $j \rightarrow \infty $ and letting $\varepsilon\rightarrow 0$, we get
\begin{equation*}
\liminf_{j \rightarrow \infty} \frac{ \left | \left \{ 1 \leq N \leq M_j : S_N( f_1, \alpha) \geq \log N  \psi( \log N ) \right \} \right |}{M_j}  = 1.
\end{equation*}
The case where $c_1$ from the definition of $f_1$ is negative can be handled under minor modifications. Analogously, one can show that the set $ \left \{ N \in \N : S_N(f, \alpha) \leq - \log(N) \psi( \log(N)) \right\} $ has upper density $1$.
\par{}
The case where $\sum_{k=1}^{\infty} \frac{1}{\psi(k)} < \infty$ can be treated analogously to the proof of Theorem \ref{ThmUpperDensDiscFunctions} by using the Denjoy--Koksma inequality with singularity (Proposition \ref{denj_koks_log}).
\par{}
To prove Theorem \ref{koks_sharp_log}, we start with a few general estimates. Observe that, for $x_1 = \frac{r}{s}$ and $q_{K-1} < N < q_K$, we have

\begin{align*} 
 \sup_{x \in [0,1)\setminus A_{N}} \left | f_1(x) + f_2(x) \right | & \ll \left | \log \left ( \min_{n < q_{K}} \| n \alpha - x_1 \| \right ) \right | \\
 & \leq \left | \log \left ( \min_{n < q_{K +1} / s}  \| n \alpha - x_1 \| \right ) \right |  \\
 & \leq \left | \log \left ( \frac{\| q_{K+s+1} \alpha \|}{s} \right ) \right |  \\
 & \ll  \log \left ( q_{K+s+2}  \right ) \\ 
 & \ll \log q_K,
\end{align*}
where we have used \eqref{Eq_size_of_q_k} and Proposition \ref{PropDistanceNalphaToq}. This shows that

\[\sup_{x \in [0,1)\setminus A_N}\lvert f(x)\rvert  \sum_{i =1}^{K}a_i
\ll \log q_K  \sum_{i =1}^{K}a_i.
\]

Further, we get
\begin{align*}
    N\left | \int_{A_{N}} f(x) dx \right| & \ll N \left | \min_{n < q_K} \| n \alpha - x_1 \| \left( \log \left( \min_{n <  q_K} \| n \alpha - x_1 \|  \right) -1  \right) \right | \\
    & \ll \frac{NK}{q_K} 
    \\& \ll \log q_K \\
    &= o(a_K),
\end{align*}
where we used that $ \min_{n < q_K} \| n \alpha - x_1 \| \leq \frac{1}{q_K}$ and $\left | \log \left( \min_{n \leq q_K} \| n \alpha - x_1 \|  \right) \right | \ll K  $ by the previous calculation. Now let $0 < r < 1$. We will show that there exists a constant $C(r) > 0$ such that the set 
\begin{equation*}
\left \{ N \in \N \ | \ S_N( f , \alpha, q ) \geq C(r)  \sup_{x \in [0,1)\setminus A_N}\lvert f(x)\rvert  \sum_{i =1}^{K(N)}a_i
+ N \left\lvert\int_{A_N} f(x)\,\mathrm{d}x\right\rvert \right\}
\end{equation*}
has upper density of at least $r$. To that end, let $(k_j)_{j \in \N}$ be the sequence of integers from the first part of this proof. Let $ A_j^{\delta'} = \left \{ 1 \leq N \leq M_j : \delta' < \frac{b_{k_j -1}(N)}{a_{k_j}} \leq 1 \right \}$ and $M_j = \left \lfloor \frac{a_{k_j}}{4s} q_{k_j -1} \right  \rfloor$ be as in the first part of this proof, where we choose $ \delta' = \delta'(r) > 0$ sufficiently small such that $ \frac{| A_j^{ \delta'} |}{M_j} \geq r$. Using \eqref{f1_estim} we get, for $N \in A_{j}^{\delta'} $, 

\[S_N(f_1,\alpha) \geq \frac{c(\delta')}{2}a_{k_j} \log q_{k_j-1}.\]

Since $\left | S_N( f_2, \alpha)  \right | \ll k_j^2 \log k_j = o(a_{k_j}\log q_{k_j-1})$, we obtain

\[S_N(f,\alpha) \geq \frac{c(\delta')}{4}a_{k_j}\log_{q_{k_j-1}} \gg \sup_{x \in [0,1)\setminus A_N}\lvert f(x)\rvert  \sum_{i =1}^{k_j}a_i
+ N \left\lvert\int_{A_N} f(x)\,\mathrm{d}x\right\rvert.\]

Thus, there exists a $C( \delta')= C(r)$ such that

\[S_N(f,\alpha) \geq  C(r) \sup_{x \in [0,1)\setminus A_N}\lvert f(x)\rvert  \sum_{i =1}^{k_j}a_i
+ N \left\lvert\int_{A_N} f(x)\,\mathrm{d}x\right\rvert\]

holds for all $N \in A_j^{\delta'}$. Since $\frac{|A_j^{\delta'}|}{M_j} \geq r$, the statement of Theorem \ref{koks_sharp_log} follows by an analogous argument as in \eqref{EqUpDens1log}. The second case in Theorem \ref{koks_sharp_log}, where we deal with the set
\[
\left \{ N \in \N \ | \ S_N( f , \alpha ,q) \leq - C(r)\sup_{x \in [0,1)\setminus A_N}\lvert f(x)\rvert  \sum_{i =1}^{K(N)}a_i
- N \left\lvert\int_{A_N} f(x) \,\mathrm{d}x\right\rvert \right\} \]
can be handled similarly.

\section*{Appendix}

In the appendix, we provide the proofs of Proposition \ref{upp_dens_wo_cond} and Proposition \ref{counterex_density1}.
\begin{proof}[Proof of Proposition \ref{upp_dens_wo_cond}]
One can easily check that the condition of $f$ being as in \eqref{counterexample_form} is invariant under rational translation, thus it suffices to prove the statement for $ q = 0$. We show that the set $ \{ N \in \N : S_N(f, \alpha) \geq \psi(\log N) \}$ has upper density $1$. We can write
\begin{equation*}
    S_N(f, \alpha ) = S_{b_{K-1} q_{K-1}}(f, \alpha) + S_{N'}(f, \alpha , b_{K-1}q_{K-1} \alpha),
\end{equation*}
where $N = b_{K-1} q_{K-1} + N'$ with $N' < q_{K-1}$. By the Denjoy--Koksma inequality in the form of \eqref{koksma_bound}, we have that 

\begin{equation*}
S_{N'}(f, \alpha , b_{K-1}q_{K-1} \alpha) \ll \sum_{i=1}^{K-1} a_i.
\end{equation*}

We analyze the dominating term $ S_{b_{K-1} q_{K-1}}(f, \alpha) $ for certain $b_{K-1}$. By Lemma \ref{prop_telescopic_part}, we have
\begin{equation}
\label{counterex_analysis}
\begin{split}
S_{b_{K-1} q_{K-1}}(f, \alpha) & = b_{K-1} \Bigg( \left\{ \frac{q_{K-1} r_1}{s_1}\right\}
+ (-1)^{K-1}\mathds{1}_{\{s_1 \mid q_{K-1}\}}
 - 2 \left\{ \frac{q_{K-1} r_2}{s_2} \right\} - 2(-1)^{K-1} \mathds{1}_{ \{ s_2 \mid q_{K-1} \} } \\
 & \qquad + \left\{\frac{q_{K-1}r_3}{s_3}\right\}+(-1)^{K-1}\mathds{1}_{ \{ s_3 \mid q_{K-1} \} }  \Bigg),
 \end{split}
\end{equation}
provided $ b_{K-1} \leq \frac{1}{2 s_1 s_2 s_3} a_K$ and $a_K$ is sufficiently large. By Lemma \ref{LemInfiniteSets}, for almost all $\alpha \in [0,1)$ and for any pair of integers $(a,b)$, the sets
\begin{equation*}
\begin{split}
    & \left \{ K \in \N \ | \ a_{K} > \psi(K) , 2 \nmid (K -1)  , q_{K -1} \equiv  a \pmod{b} , \sum_{i=1}^{K-1} a_i \leq 2 K \log K \right\}, \\
    & \left \{ K \in \N \ | \ a_{K} > \psi(K) , 2 \mid (K-1) , q_{K-1}  \equiv  a \pmod{b} , \sum_{i=1}^{K-1} a_i \leq 2 K \log K \right\}
\end{split} 
\end{equation*}
both contain infinitely many integers $K$. In the upcoming case distinction we will consider different choices of $a$ and $b$.
\par{}
Case 1: $s_1 \nmid s_2$: The congruence relation $q_{K-1} \equiv s_2 \mod{s_1 s_2}$ ensures that $ s_2 \mid q_{K-1} $ and $ s_1 \nmid q_{K-1}$ since $s_1 \nmid s_2$ by assumption. Thus, \eqref{counterex_analysis} gives us
\begin{align*}
S_{ b_{K-1} q_{K-1}}(f,\alpha)  &= b_{K-1} \Bigg( \underbrace{\left\{\frac{q_{K-1}r_1}{s_1}\right\}}_{ \geq 0} + 2 +  \underbrace{\left\{\frac{q_{K-1} r_3}{s_3}\right\} - \mathds{1}_{\{s_3 \mid q_{K-1} \}}}_{\geq -1} \Bigg) \\
& \geq b_{K-1}.
\end{align*}

Case 2: We assume $s_3 \nmid s_2$. Under the congruence conditions $q_{K-1} \equiv s_2 \mod{s_2 s_3}$ and $2 \nmid (K-1)$, we obtain 
\begin{equation*}
\begin{split}
S_{b_{K-1} q_{K-1}}(f, \alpha)  & = b_{K-1} \Bigg( \left\{ \frac{q_{K-1} r_1}{s_1}\right\}
- \mathds{1}_{\{s_1 \mid q_{K-1}\}}
  + 2  + \left\{\frac{q_{K-1}r_3}{s_3}\right\}  \Bigg) \\
  & \geq b_{K-1}.
 \end{split}
\end{equation*}

Case 3: If $s = s_1 = s_2 = s_3$, then we use the congruence conditions $2 \nmid (K-1)$ and $q_{K-1} \equiv (r_2)^{-1} \pmod{s}$
to show
\begin{align*}
S_{ b_{K-1}q_{K-1}}(f,\alpha) & = b_{K-1} \Bigg( \left\{ \frac{q_{K-1} r_1}{s}\right\}
 - 2 \left\{ \frac{q_{K-1} r_2}{s} \right\} + \left\{\frac{q_{K-1}r_3}{s}\right\}  \Bigg) \\ 
& = b_{K-1} \left(  \left\{ \frac{q_{K-1} r_1}{s}\right\}
- \frac{2}{s} +  \left\{\frac{q_{K-1} r_3}{s} \right\} \right) \\
& \geq \frac{2}{s} b_{K-1}.
\end{align*}

The latter inequality holds since $r_1,r_3$ are distinct from $r_2$, and thus $\left\{\frac{q_{K-1} r_1}{s}\right\} + \left\{\frac{q_{K-1} r_3}{s}\right\} \geq \frac{4}{s}$.

Case 4: $s_1 \mid s_2$ and $ s_3\mid s_2$, but $s_1 \neq s_2 $ or $s_3 \neq s_2$: Without loss of generality, we assume $s_1 \neq s_2$. We use the congruence conditions $2 \mid (K-1) $ and $  q_{K-1} \equiv a s_1 \pmod{s_2}$ with $ a \equiv r_2^{-1} \mod{ \frac{s_2}{s_1}} $ (which is possible since $ \gcd (r_2, s_2) = 1$). We obtain
\begin{align*}
  S_{b_{K-1}q_{K-1}}(f,\alpha) & =  b_{K-1} \Bigg(  1
 - 2 \left\{ \frac{q_{K-1} r_2}{s_2} \right\} + \left\{\frac{q_{K-1}r_3}{s_3}\right\}+ \mathds{1}_{ \{ s_3 \mid q_{K-1} \} }  \Bigg) \\
  & = b_{K-1} \Bigg(  1 - 2\left\{\frac{ s_1}{s_2}\right\} + \left\{\frac{q_{K-1} r_3}{s_3}\right\}+ \mathds{1}_{\{s_3 \mid q_{K-1} \}} \Bigg) \\
  & \geq b_{K-1} \Bigg(  \left\{\frac{q_{K-1} r_3}{s_3}\right\}+ \mathds{1}_{\{s_3 \mid q_{K-1} \}} \Bigg) \\
  & \geq \frac{1}{s_3} b_{K-1}.
\end{align*}

The second last inequality follows from the congruence relation $   q_{K-1} \equiv a s_1 \pmod{s_2} $ and $ 2  \left\{\frac{ s_1}{s_2}\right\} \leq 1 $, where the latter holds since $s_1 \mid s_2$ and $s_1 \neq s_2$. In the last line we used that if $ \mathds{1}_{\{s_3 \mid q_{K-1} \}} = 0 $, then $  \left\{\frac{q_{K-1} r_3}{s_3}\right\} \geq \frac{1}{s_3}$.

Thus, in either case, there exists a constant $ C > 0 $ such that, for $N \in \N$ with $b_{K-1} \leq \frac{1}{2 s_1 s_2 s_3} a_K$, we have
\begin{equation*}
    S_N( f, \alpha) \geq C  b_{K-1} +  O \left( \sum_{i=1}^{K-1} a_i \right).
\end{equation*}
The remaining part of the proof can be argued in the same way as it is done in the proof of Theorem \ref{ThmUpperDensDiscFunctionsRefined}. The set $ \{ N \in \N : S_N(f, \alpha) \leq - \psi(\log N) \}$ can be handled analogously.
\end{proof}

\begin{proof}[Proof of Proposition \ref{counterex_density1}]
We will show that for almost every $\alpha \in [0,1)$, there exists a $ \delta > 0$ with
\[
\liminf_{M \rightarrow \infty} \frac{\left | \{1 \leq N \leq M: \lvert S_N(f,\alpha) \rvert \ll \log N \log \log N \} \right | }{M}
\geq \delta.
\]
By choosing $\psi(k) := k \log k \log \log (k +10)$, this implies that 
\begin{equation*}
\limsup_{M \rightarrow \infty} \frac{\left | \{1 \leq N \leq M: \lvert S_N(f,\alpha) \rvert \geq  \psi(\log N) \} \right | }{M}
\leq 1 -\delta.
\end{equation*}
Fixing $M \in \mathbb{N}$, there is exactly one $K \in \mathbb{N}$ such that $q_{K-1} \leq M < q_{K}$.
Let $K_0 = \argmax_{k \leq K}a_k$ (if the maximum is not unique, we can choose an arbitrary one among the maximizers). We define
\[A_{M}^{\delta} := \left\{ \sqrt{q_{K-1}} \leq N \leq M:   b_{K_0-1}(N) \leq  \delta   a_{K_0} \right\},\]
where $\delta > 0$ is a small constant specified later. In the following, we will show that for any $N \in A_M^{\delta}$, we have $\lvert S_N(f,\alpha)\rvert \ll K \log K$.
Writing $N = \sum_{\ell=0}^{K-1} b_{\ell}q_{\ell}$ in its Ostrowski expansion, we obtain the decomposition 

\[S_N(f,\alpha) = S_{N'}( f , \alpha) +  S_{b_{K_0-1} q_{K_0-1}}(f, \alpha, N' \alpha ) + S_{N''}(f, \alpha, (N' + b_{K_0 -1}q_{K_0 -1}) \alpha),
\]

with $N' = \sum_{ \ell =K_0 }^{K-1} b_{\ell} q_{\ell}$ and $N'' = \sum_{\ell=0}^{K_0 -2} b_{ \ell} q_{\ell}$. By the Denjoy--Koksma inequality \eqref{koksma_bound}, we can bound $S_{N'}( f, \alpha) $ by
\begin{align*}
\left | S_{N'}( f, \alpha) \right | 
& \ll  \sum_{i=K_0+1}^{K} a_i \\
& \ll K \log K,
\end{align*}
where we used \eqref{Eq_trimmed_sum} in the second line. Analogously, one obtains the same bound for $S_{N''}(f, \alpha, (N' + b_{K_0 -1}q_{K_0 -1}) \alpha)$, i.e., we get
\begin{equation*}
    \left | S_{N''}(f, \alpha, (N' + b_{K_0 -1}q_{K_0 -1}) \alpha) \right |  \ll K \log K.
\end{equation*}

 We now turn our attention to $ S_{b_{K_0-1} q_{K_0-1}}(f, \alpha, N' \alpha )$, where we will show $ S_{b_{K_0-1} q_{K_0-1}}(f, \alpha, N' \alpha ) = 0$.
Indeed, an analogous analysis to the proof of Lemma \ref{prop_telescopic_part} shows that there exists a $ \delta  > 0$ such that, for any $ b_{K_0 -1} \leq \delta  a_{k_0} $, it holds

\[S_{b_{K_0 -1} q_{K_0 -1}}(f,\alpha)
= \left(\left\{\frac{ q_{K_0 -1} u}{w}\right\} + \left\{\frac{q_{K_0 -1}(1-u)}{w}\right\}\right)
     - \left(\left\{\frac{q_{K_0 -1}v}{w}\right\} + \left\{\frac{q_{K_0 -1}(1-v)}{w}\right\}\right).
\]

Regardless of the congruence class of $ q_{K_0 -1}$  modulo $w$, the expression above equals $0$. In total, we have shown that, for all $N \in A_M^{\delta}$, we get the asymptotic bound
\begin{equation*}
    \left | S_N(f, \alpha) \right | \ll K \log K \ll \log N \log \log N,
\end{equation*}
where the last estimate uses that $N \geq \sqrt{q_{K-1}}$, which holds by the definition of $A_M^{\delta}$. This leads to
\begin{align*}
\liminf_{M \rightarrow \infty} \frac{\left | \{1 \leq N \leq M: \lvert S_N(f,\alpha) \rvert \ll \log N \log \log N \} \right | }{M}
& \geq  \liminf_{M \rightarrow \infty} \frac{\left | \{ N \in A_M^{\delta} : \lvert S_N(f,\alpha) \rvert \ll \log N \log \log N \} \right | }{M}  \\
& \geq \liminf_{M \rightarrow \infty} \frac{ | A_M^{\delta}| }{M} \\
& \geq \delta.
\end{align*}
This finishes the proof.
\end{proof}

\subsection*{Acknowledgements} 
We would like to thank Bence Borda for many valuable discussions. LF and MH were supported by the Austrian Science Fund (FWF) Project P 35322 \textit{Zufall und Determinismus in Analysis und Zahlentheorie}.

\author{Lorenz Fr\"uhwirth}
{\footnotesize

Graz University of Technology

Steyrergasse 30, 8010 Graz, Austria

Email: \texttt{fruehwirth@math.tugraz.at}

}
\vspace{5mm}

\author{Manuel Hauke}
{\footnotesize 

University of York

Department of Mathematics

YO10 5DD York, United Kingdom

Email: \texttt{hauke@math.tugraz.at; manuel.hauke@york.ac.uk}}

\end{document}